\documentclass[12pt,twoside]{amsart}
\usepackage[latin1]{inputenc}

\usepackage{latexsym}
\usepackage{exscale}

\usepackage{amsmath,amsfonts,amssymb,amsthm,fullpage,bbm}

\usepackage{fancyhdr,graphicx,color}

\numberwithin{equation}{section}

\newtheorem{thm}{Theorem}[section]

\newtheorem{lem}[thm]{Lemma}

\newtheorem{cor}[thm]{Corollary}

\newtheorem{prop}[thm]{Proposition}

\newtheorem{rem}[thm]{Remark}

\newtheorem{dfn}[thm]{Definition}
\newcommand{\diam}{diam}
\newcommand{\aver}[1]{-\hskip-0.46cm\int_{#1}}

\newcommand{\ind}{1\hspace{-2.5 mm}{1}}

\DeclareMathOperator{\supp}{supp}
\DeclareMathOperator{\Lip}{Lip}
\makeatletter
\renewcommand\@biblabel[1]{#1.}
\makeatother
\begin{document}
\allowdisplaybreaks
\title{Real interpolation of Sobolev spaces}
\author{Nadine BADR}
\address{N.Badr
\\
Universit\'e de Paris-Sud, UMR du CNRS 8628
\\
91405 Orsay Cedex, France} \email{nadine.badr@math.u-psud.fr}

%\date{\today}
\begin{abstract}
We prove that $W^{1}_{p}$ is an interpolation space between $W^{1}_{p_{1}}$ and $W^{1}_{p_{2}}$  for $p>q_{0}$ and $1\leq p_{1}<p<p_{2}\leq \infty$  on some classes of manifolds and general metric spaces, where $q_{0}$ depends on our hypotheses. 
\end{abstract}
\maketitle
\tableofcontents

\section{Introduction}
Do the Sobolev spaces $W^{1}_{p}$ form a  real interpolation scale for $1<p <\infty$?
The aim of the present work is to provide a positive answer for Sobolev spaces on some metric spaces.
Let us state here our main theorems for non-homogeneous Sobolev spaces (resp. homogeneous Sobolev spaces) on Riemannian manifolds. 
%For other metric-measure spaces see sections 7 and 8.
\begin{thm}\label{IS}
Let $M$ be a complete non-compact Riemannian manifold satisfying the local doubling property $(D_{loc})$ and a local Poincar\'{e} inequality $(P_{qloc})$, for some $1\leq q <\infty$. Then for $1\leq r \leq q <p <\infty$, $W_{p}^{1}$ is a real interpolation space between $W_{r}^{1}$ and  $W_{\infty}^{1}$.
\end{thm}
\noindent To prove Theorem \ref{IS}, we characterize the $K$-functional of real interpolation for non-homogeneous Sobolev spaces:
 \begin{thm}\label{EK}
 Let $M$ be as in Theorem \ref{IS},
\begin{itemize}
\item[1.]  There exists $C_{1}>0$ such that for all  $f \in
 W^{1}_{r}+W^{1}_{\infty}$ and all $t>0$ we have
 \begin{equation*} 
 %\tag{$\ast_{\textrm{loc}}$}
K(f,t^{\frac{1}{r}},W^{1}_{r},W^{1}_{\infty})\geq C_{1} t^{\frac{1}{r}}\Bigl(|f|^{r **\frac{1}{r}}(t)+|\nabla
f|^{r**\frac{1}{r}}(t)\Bigr);\
\end{equation*}
\item[2.] For $r\leq q\leq p<\infty$, there is $C_2>0$ such that for all $f\in W_{p}^{1}$ 
\begin{equation*}K(f,t^{\frac{1}{r}},W^{1}_{r},W^{1}_{\infty})\leq C_{2}t^{\frac{1}{r}}
\Bigl(|f|^{q**\frac{1}{q}}(t)+|\nabla
f|^{q**\frac{1}{q}}(t)\Bigr)
\end{equation*}  
\end{itemize}
In the special case $r=q$, we obtain the upper bound of $K$ in point 2. for every $f\in W_q^1+W_\infty^1$ and hence get a true characterization of $K$.
\end{thm}
\noindent The proof of this theorem relies on a Calder\'{o}n-Zygmund decomposition for Sobolev functions (Proposition \ref{CZ}).
 \\
 Above and from now on,  $|g|^{q**\frac{1}{q}}$ means $(|g|^{q**})^{\frac{1}{q}}$ --see section 2 for the definition of $g^{**}$--.
 
 The reiteration theorem (\cite{bennett}, Chapter 5, Theorem 2.4 p.311) and an improvement result for the exponent of a Poincar\'e inequality due to Keith and Zhong yield a more general version of Theorem \ref{IS}. Define $q_{0}=\inf\left\lbrace q\in [1,\infty[:\, (P_{qloc})\, \textrm {holds }\right\rbrace$. 

\begin{cor} \label{CIS}For $1\leq p_{1}<p<p_{2}\leq \infty$ with $p>q_{0}$, $W_{p}^{1}$ is a real interpolation space between $W_{p_1}^{1}$ and $W_{p_2}^{1}$. More precisely 
$$ W_{p}^{1}=(W_{p_1}^1,W_{p_2}^{1})_{\theta,p}
$$
where $0<\theta<1$ such that $\frac{1}{p}=\frac{1-\theta}{p_1}+\frac{\theta}{p_2}$. 
\end{cor}
However, if $p\leq q_0$, we only know that that $(W_{p_1}^{1},W_{p_{2}}^{1})_{\theta,p}\subset W_{p}^{1}$. 
\\
 %5. The consequence for the interpolation problem is stated as follows.
%we prove below, in section 5, a weaker form of Theorem \ref{EK} that allows us to draw the following conclusion.
For the homogeneous Sobolev spaces, a weak form of Theorem \ref{EK} is available. This result is presented in section 5. The consequence for the interpolation problem is stated as follows.
 \begin{thm}\label{IHS}  Let $M$ be a complete non-compact Riemannian manifold satisfying the global doubling property $(D)$ and a global Poincar\'{e} inequality $(P_{q})$ for some $1\leq q <\infty$. Then, for $1\leq r\leq q<p<\infty$, $\overset{.}{W_{p}^{1}}$ is an interpolation space between $\overset{.}{W_{r}^{1}}$ and $\overset{.}{W_{\infty}^{1}}$. 
\end{thm}
Again, the reiteration theorem implies another version of Theorem \ref{IHS}; see section 5 below. 
\\

For $\mathbb{R}^{n}$ and the non-homogeneous Sobolev spaces, our interpolation result follows from the leading work of Devore-Scherer \cite{devore1}. The method of \cite{devore1} is based on spline functions. Later, simpler proofs were given by Calder\'{o}n-Milman \cite{calderon1} and Bennett-Sharpley \cite{bennett}, based on the Whitney extension and covering theorems. Since $\mathbb{R}^{n}$ admits $(D)$ and $(P_1)$, we recover this result by our method. Moreover, applying Theorem \ref{IHS}, we obtain the interpolation of the homogeneous Sobolev spaces on $\mathbb{R}^{n}$. Notice that this result is not covered by the existing references.
\\
The interested reader may find a wealth of examples of spaces satisfying doubling and Poincar\'e inequalities --to which our results apply-- in
%For other examples  -where our result applies-, the reader is referred to
 \cite{ambrosio1}, \cite{auscher2}, \cite{franchi3}, \cite{franchi7}, \cite{hajlasz4}. 
\\

Some comments about the generality of Theorem \ref{IS}- \ref{IHS} are in order. First of all, completeness of the Riemannian manifold is not necessary (see Remark \ref{C}). Also, our technique can be adapted to more general metric-measure spaces, see sections 7-8. Finally it is possible to build examples where interpolation \textit{without} a Poincar\'e inequality \textit{is} possible. The question of the necessity of a Poincar\'e inequality for a general statement arises. This is discussed in the Appendix. 
\\

The initial motivation of this work was to provide an answer for the interpolation question for $\overset{.}{W^{1}_{p}}$. This problem was explicitly posed in \cite{auscher1}, where the authors interpolate inequalities of type $\|\Delta^{\frac{1}{2}}f\|_{p}\leq C_{p}\|\,|\nabla f|\,\|_{p}$ on Riemannian manifolds.
\\

Let us briefly describe the structure of this paper. In section 2 we review the notions of a doubling property as well as the real $K$ interpolation method. In sections 3 to 5, we study in detail the interpolation of Sobolev spaces in the case of a complete non-compact Riemannian manifold $M$ satisfying $(D)$ and $(P_{q})$ (resp. $(D_{loc})$ and $(P_{qloc})$). We briefly mention the case where $M$ is a compact manifold in section 6. In section 7, we explain how our results extend to more general metric-measure spaces. We apply this interpolation result to Carnot-Carath\'eodory spaces, weighted Sobolev spaces and to Lie groups in section 8. Finally, the Appendix is devoted to an example where the Poincar\'e inequality is not necessary to interpolate Sobolev spaces.
%Section 11 is devoted for some examples to which our results apply.
\\

\noindent\textit{Acknowledgements.} I am deeply indebted to my Ph.D advisor P. Auscher, who suggested to study the topic of this paper, and for his constant encouragement and useful advices. Also I am thankful to P. Hajlasz for his interest in this work and M. Milman for communicating me his paper with J. Martin \cite{martin}. Finally, I am also grateful to G. Freixas, with whom I had interesting discussions regarding this work. 
\section{Preliminaries}
Throughout this paper we will denote by $\ind_{E}$ the characteristic function of
 a set $E$ and $E^{c}$ the complement of  $E$. If $X$ is a metric space, $\Lip$ will be  the set of real Lipschitz functions on $X$ and $\Lip_{0}$ the set of real, compactly supported Lipschitz functions on $X$. For a ball $B$ in a metric space, $\lambda B$  denotes the ball co-centered with $B$ and with radius $\lambda$ times that of $B$. Finally, $C$ will be a constant
that may change from an inequality to another and we will use $u\sim
v$ to say that there exists two constants $C_{1}$, $C_{2}>0$ such that $C_{1}u\leq v\leq
C_{2}u$.
\subsection{The doubling property}
By a metric-measure space, we mean a triple $(X,d,\mu)$ where $(X,d)$ is a metric space and $\mu$ a non negative Borel measure. Denote by $B(x, r)$ the open ball of center $x\in X $ and radius $r>0$.
\begin{dfn} Let $(X,d,\mu)$ be a metric-measure space. One says that $X$ satisfies the  local doubling property $(D_{loc})$ if there exist constants $r_{0}>0$, $0<C=C(r_{0})<\infty$, such that for all $x\in X,\, 0<r< r_{0} $ we have
\begin{equation*}\tag{$D_ {loc}$}
\mu(B(x,2r))\leq C \mu(B(x,r)).
\end{equation*}
Furthermore $X$ satisfies a global doubling property or simply doubling property $(D)$ if one can take $r_{0}=\infty$.
We also say that $\mu$ is a locally (resp. globally) doubling Borel measure.
\end{dfn}
\noindent Observe that if $X$ is a metric-measure space satisfying $(D)$ then
$$
 \diam(X)<\infty\Leftrightarrow\,\mu(X)<\infty\,\textrm{ (\cite{ambrosio1})} . 
 $$
\begin{thm}[Maximal theorem]\label{MIT} (\cite{coifman2})
Let $(X,d,\mu)$ be a metric-measure space satisfying $(D)$. Denote by $\mathcal{M}$ the uncentered Hardy-Littlewood maximal function over open balls of $X$ defined by
 $$
 \mathcal{M}f(x)=\underset{B:x\in B}{\sup}|f|_{B}
 $$ 
 where $ \displaystyle f_{E}:=\aver{E}f d\mu:=\frac{1}{\mu(E)}\int_{E}f d\mu.$
Then

\begin{itemize}
\item[1.] $\mu(\left\lbrace x:\,\mathcal{M}f(x)>\lambda\right\rbrace)\leq \frac{C}{\lambda}\int_{X}|f| d\mu$ for every $\lambda>0$;
\item[2.] $\|\mathcal{M}f\|_{L_{p}}\leq C_{p} \|f\|_{L_{p}}$, for $1<p\leq\infty$.
\end{itemize}
\end{thm}
\subsection{The $K$-method of real interpolation} The reader can refer to \cite{bennett}, \cite{bergh} for details on the development of this theory. Here we only recall the essentials to be used in the sequel. 

Let $A_{0}$, $A_{1}$ be  two normed vector spaces embedded in a topological Hausdorff vector space $V$. For each  $a\in A_{0}+A_{1}$ and $t>0$, we define the $K$-functional of interpolation by
$$
K(a,t,A_{0},A_{1})=\displaystyle \inf_{a
=a_{0}+a_{1}}(\| a_{0}\|_{A_{0}}+t\|
a_{1}\|_{A_{1}}).
$$

For $0<\theta< 1$, $1\leq q\leq \infty$, we denote by $(A_{0},A_{1})_{\theta,q}$ the interpolation space between $A_{0}$ and $A_{1}$:
\begin{displaymath}
	(A_{0},A_{1})_{\theta,q}=\left\lbrace a \in A_{0}+A_{1}:\|a\|_{\theta,q}=\left(\int_{0}^{\infty}(t^{-\theta}K(a,t,A_{0},A_{1}))^{q}\,\frac{dt}{t}\right)^{\frac{1}{q}}<\infty\right\rbrace.
\end{displaymath}
It is an exact interpolation space of exponent $\theta$ between $A_{0}$ and $A_{1}$, see \cite{bergh}, Chapter II.
\begin{dfn}
Let $f$  be a measurable function on a measure space $(X,\mu)$. The decreasing rearrangement of $f$ is the function $f^{*}$ defined for every $t\geq 0$ by
$$
f^{*}(t)=\inf \left\lbrace\lambda :\, \mu (\left\lbrace x:\,|f(x)|>\lambda\right\rbrace)\leq
t\right\rbrace.
$$
The maximal decreasing rearrangement of
$f$ is the function $f^{**}$ defined for every $t>0$ by
$$
f^{**}(t)=\frac{1}{t}\int_{0}^{t}f^{*}(s) ds.
$$
\end{dfn}
 It is known that $(\mathcal{M}f)^{*}\sim f^{**}$ and $\mu (\left\lbrace x:\, |f(x)|>f^{*}(t)\right\rbrace)\leq t$ for all $t>0$.
We refer to \cite{bennett}, \cite{bergh}, \cite{calderon2} for other properties of $f^{*}$ and $f^{**}$.
\\

We conclude the preliminaries by quoting the following classical result (\cite{bergh} p.109):
\begin{thm}\label{IK} Let $(X,\mu)$ be a measure space where $\mu$ is a totally $\sigma$-finite 
positive measure. Let $f\in L_{p}+L_{\infty}$, $0<p<\infty$ where
$L_{p}=L_{p}(X,d\mu)$. We then have
\begin{itemize}
\item[1.] $K(f,t,L_{p},L_{\infty})\sim\Bigl(\int_{0}^{t^{p}}(f^{*}(s))^{p}ds\Bigr)^\frac{1}{p}$ and equality  holds for $p=1$;
\item[2.] for $0<p_{0}<p<p_{1}\leq\infty$, $(L_{p_{0}},L_{p_{1}})_{\theta,p}=L_{p}$ with equivalent norms, where $\displaystyle\frac{1}{p}=\frac{1-\theta}{p_{0}}+\frac{\theta}{p_{1}}$
with $0<\theta<1$.
\end{itemize}
\end{thm}
\section{Non-homogeneous Sobolev spaces on Riemannian manifolds}
In this section $M$ denotes a complete non-compact Riemannian manifold. We write $\mu$ for the Riemannian measure on $M$, $\nabla$ for the Riemannian gradient, $|\cdot|$ for the length on the tangent space (forgetting the subscript $x$ for simplicity) and $\|\cdot\|_{p}$ for the norm on $ L_{p}(M,\mu)$, $1 \leq p\leq +\infty.$ Our goal is to prove Theorem \ref{EK}.
\subsection{Non-homogeneous Sobolev spaces}
\begin{dfn}[\cite{aubin1}]\label{DNH} Let $M$ be a $C^{\infty}$ Riemannian manifold
 of dimension $n$. Write $E^{1}_{p}$ for the vector space of $C^{\infty}$ functions $\varphi
$ such that $\varphi $ and $|\nabla\varphi|\in L_p,
\,1\leq p< \infty$. We define the Sobolev space $W^{1}_{p}$ as the completion of  $E^{1}_{p}$ for the norm
$$
\|\varphi\|_{W^{1}_{p}}=\|\varphi\|_{p}+\|\,|\nabla\varphi|\,\|_{p}.
$$
We denote $W^{1}_{\infty}$ for the set of all bounded Lipschitz functions on $M$.
\end{dfn}
\begin{prop}\label{CDW} (\cite{aubin1}, \cite{goldshtein}) Let $M$ be a complete Riemannian manifold. Then $ C^{\infty}_{0}$ and in particular $Lip_{0}$ is dense in $W^{1}_{p}$ for $1\leq p<\infty$.
 \end{prop}
\begin{dfn}[Poincar\'{e} inequality on $M$] We say that a complete Riemannian manifold $M$ admits \textbf{a local Poincar\'{e} inequality $(P_{qloc})$} for some $1\leq q<\infty$ if there exist constants $r_{1}>0,\,C=C(q,r_{1})>0$ such that, for every function $f\in \Lip_{0}$ and every ball $B$ of $M$ of radius $0<r<r_{1}$, we have
\begin{equation*}\tag{$P_{qloc}$}
\aver{B}|f-f_{B}|^{q} d\mu \leq C r^{q} \aver{B}|\nabla f|^{q}d\mu.
\end{equation*}
$M$ admits a global Poincar\'{e} inequality $(P_{q})$ if we can take $r_{1}=\infty$ in this definition.
\end{dfn}
\begin{rem} By density of $C_{0}^{\infty}$ in $W_{p}^{1}$, we can replace $\Lip_{0}$ by $C_{0}^{\infty}$.
\end{rem}
\subsection[Estimation of $K$]{Estimation of the $K$-functional of interpolation}
In the first step, we prove Theorem \ref{EK} in the global case. This will help us to understand the proof of the more general local case.
\subsubsection{The global case}\label{CG}
 Let $M$ be a complete Riemannian manifold satisfying $(D)$ and $(P_{q})$, for some $1\leq q <\infty$. Before we prove  Theorem \ref{EK}, we make a Calder\'{o}n-Zygmund decomposition for Sobolev functions inspired by the one done in \cite{auscher1}. To achieve our aims, we state it for more general spaces (in \cite{auscher1}, the authors only needed the decomposition for the functions $f$ in $C^{\infty}_{0}$). This will be the principal tool in the estimation of the functional $K$.
\begin{prop}[Calder\'{o}n-Zygmund lemma for Sobolev functions]\label{CZ} Let $M$ be a complete non-compact Riemannian manifold satisfying $(D)$.
Let $1\leq q<\infty$ and assume that $M$ satisfies $(P_{q})$. Let $q\leq p<\infty$, $f \in W^{1}_{p}$ and $\alpha>0$. Then one can find a collection of balls $(B_{i})_{i}$, functions $b_{i}\in W_{q}^{1}$ and a Lipschitz function $g$ such that the following properties hold:
\begin{equation}
f = g+\sum_{i}b_{i} \label{df}
\end{equation}
\begin{equation}
|g(x)|\leq C\alpha\,\textrm{ and }\,|\nabla g(x)|\leq C\alpha\quad \mu-a.e\; x\in M \label{eg}
\end{equation}
\begin{equation}
\supp b_{i}\subset B_{i}, \,\int_{B_{i}}(|b_{i}|^{q}+|\nabla b_{i}|^{q})d\mu\leq C\alpha^{q}\mu(B_{i})\label{eb}
\end{equation}
\begin{equation}
\sum_{i}\mu(B_{i})\leq C\alpha^{-p}\int (|f|+|\nabla f|)^{p} d\mu
\label{eB}
\end{equation}
\begin{equation}
\sum_{i}\chi_{B_{i}}\leq N \label{rb}.
\end{equation}
The constants $C$ and $N$ only depend on $q$, $p$ and on the constants in $(D)$ and $(P_{q})$.
\end{prop}
\begin{proof} Let  $f\in W_{p}^{1}$, $\alpha>0$. Consider
$\Omega=\left\lbrace x \in M : \mathcal{M}(|f|+|\nabla f|)^{q}(x)>\alpha^{q}\right\rbrace$. If $\Omega=\emptyset$, then set
$$
 g=f\;,\quad b_{i}=0 \, \text{ for all } i
$$
so that (\ref{eg}) is satisfied according to the Lebesgue differentiation theorem. Otherwise the maximal theorem --Theorem \ref{MIT}-- gives us
\begin{align}
	\mu(\Omega)&\leq C\alpha^{-p}\|(|f|+ |\nabla f|)^{q}\|_{\frac{p}{q}}^{\frac{p}{q}} \nonumber\\
			& \leq C \alpha^{-p} \Bigr(\int | f|^{p} d\mu +\int |\nabla f|^{p} d\mu\Bigl) \label{mO}
\\
			&<+\infty. \nonumber
\end{align}
 In particular $\Omega \neq M$ as $\mu(M)=+\infty$. Let $F$ be the complement of $\Omega$. Since $\Omega$ is an open set distinct of $M$, let
$(\underline{B_{i}})$ be a Whitney decomposition of $\Omega$ (\cite{coifman1}). That is, the balls  $\underline{B_{i}}$ are pairwise disjoint and there exist two constants $C_{2}>C_{1}>1$, depending only
on the metric, such that
\begin{itemize}
\item[1.] $\Omega=\cup_{i}B_{i}$ with $B_{i}=
C_{1}\underline{B_{i}}$ and the balls $B_{i}$ have the bounded overlap property;
\item[2.] $r_{i}=r(B_{i})=\frac{1}{2}d(x_{i},F)$ and $x_{i}$ is 
the center of $B_{i}$;
\item[3.] each ball $\overline{B_{i}}=C_{2}B_{i}$ intersects $F$ ($C_{2}=4C_{1}$ works).
\end{itemize}
For $x\in \Omega$, denote $I_{x}=\left\lbrace i:x\in B_{i}\right\rbrace$. By the bounded overlap property of the balls $B_{i}$, we have that $\sharp I_{x} \leq N$. Fixing $j\in I_{x}$ and using the properties of the $B_{i}$'s, we easily see that $\frac{1}{3}r_{i}\leq r_{j}\leq 3r_{i}$ for all $i\in I_{x}$. In particular, $B_{i}\subset 7B_{j}$ for all $i\in I_{x}$.

Condition (\ref{rb}) is nothing but the bounded overlap property of the $B_{i}$'s  and (\ref{eB}) follows from (\ref{rb}) and  (\ref{mO}). The doubling property and the fact that $\overline{B_{i}} \cap F
\neq \emptyset$ yield
\begin{equation}\label{f}
\int_{B_{i}} (|f|^{q}+|\nabla f|^{q})d\mu  \leq
\int_{\overline{B_{i}}} (|f|+|\nabla f|)^{q} d\mu
\leq \alpha^{q} \mu(\overline{B_{i}})
\leq C \alpha^{q}\mu(B_{i}).
\end{equation}

Let us now define the functions $b_{i}$. Let $(\chi_{i})_{i}$ be  a partition of unity of $\Omega$ subordinated to the covering $(\underline{B_{i}})$, such that for all $i$, $\chi_{i}$ is a Lipschitz function supported in $B_{i}$ with
$\displaystyle\|\,|\nabla \chi_{i}|\, \|_{\infty}\leq
\frac{C}{r_{i}}$. To this end it is enough to choose $\displaystyle\chi_{i}(x)=
\psi(\frac{C_{1}d(x_{i},x)}{r_{i}})\Bigl(\sum_{k}\psi(\frac{C_{1}d(x_{k},x)}{r_{k}})\Bigr)^{-1}$, where $\psi$ is a smooth function, $\psi=1$ on $[0,1]$, $\psi=0$
on $[\frac{1+C_{1}}{2},+\infty[$ and $0\leq \psi\leq 1$. 	
We set $b_{i}=(f-f_{B_{i}})\chi_{i}$. It is clear that $\supp b_{i} \subset B_{i}$.
Let us estimate $\int_{B_{i}} |b_{i}|^{q} d\mu$ and
$\int_{B_{i}} |\nabla b_{i}|^{q} d\mu$. We have
\begin{align*}
\int_{B_{i}} |b_{i}|^{q} d\mu
&=\int_{B_{i}} |(f-f_{B_{i}})\chi_{i}|^{q} d\mu
\\
&\leq
C(\int_{B_{i}}|f|^{q}d\mu+\int_{B_{i}}|f_{B_{i}}|^{q} d\mu)
\\
&\leq C\int_{B_{i}}|f|^{q} d\mu
\\
&\leq C \alpha^{q} \mu(B_{i}).
\end{align*}
We applied Jensen's inequality in the second estimate, and (\ref{f}) in the last one. Since $\nabla\Bigl((f-f_{B_{i}})\chi_{i}\Bigr)=\chi_{i}\nabla f
+(f-f_{B_{i}})\nabla\chi_{i}$, the Poincar\'e inequality $(P_{q})$ and (\ref{f}) yield
\begin{align*}
\int_{B_{i}}|\nabla b_{i}|^{q}d\mu
 &\leq
C\left(\int_{B_{i}}|\chi_{i}\nabla f|^{q}d\mu
+\int_{B_{i}}|f-f_{B_{i}}|^{q}|\nabla \chi_{i}|^{q}d\mu \right)
\\
&\leq C\alpha^{q}\mu(B_{i})+ C\frac{C^{q}}{r_{i}^{q}} r_{i}^{q}\int_{B_{i}}|\nabla
f|^{q}d\mu
\\
&\leq C\alpha^{q}\mu(B_{i}).
\end{align*}
Therefore (\ref{eb}) is proved. 

Set $ \displaystyle g=f-\sum_{i}b_{i}$. Since the sum is locally finite on $\Omega$,  $g$ is defined  almost everywhere on $M$ and $g=f$ on $F$. Observe that $g$ is a locally integrable function on $M$. Indeed, let $\varphi\in L_{\infty}$ with compact support. Since $d(x,F)\geq r_{i}$ for $x \in \supp \,b_{i}$, we obtain
\begin{equation*} \int\sum_{i}|b_{i}|\,|\varphi|\,d\mu \leq
\Bigl(\int\sum_{i}\frac{|b_{i}|}{r_{i}}\,d\mu\Bigr)\,\sup_{x\in
M}\Bigl(d(x,F)|\varphi(x)|\Bigr)\quad
\end{equation*}
and
\begin {align*}
\int \frac{|b_{i}|}{r_{i}}d\mu
&=\int_{B_{i}}\frac{|f-f_{B_{i}}|}{r_{i}}\chi_{i}d\mu
\\
&\leq \Bigl(\mu(B_{i})\Bigr)^{\frac{1}{q'}}
\Bigl(\int_{B_{i}}|\nabla f|^{q} d\mu\Bigr)^{\frac{1}{q}}
\\
&\leq C\alpha\mu(B_{i}).
\end{align*}
We used the H\"{o}lder inequality, $(P_{q})$ and
the estimate (\ref{f}), $q'$ being the conjugate of $q$. Hence 
$ \displaystyle \int\sum_{i}|b_{i}||\varphi|d\mu \leq
C\alpha\mu(\Omega) \sup_{x\in M
}\Bigl(d(x,F)|\varphi(x)|\Bigr)$. Since $f\in L_ {1,loc}$, we deduce that $g\in L_{1,loc}$. (Note that since $b\in L_{1}$ in our case, we can say directly that $g\in L_{1,loc}$. However, for the homogeneous case --section 5-- we need this observation  to conclude that $g\in L_{1,loc}$.) 
It remains to prove (\ref{eg}). Note that $\displaystyle \sum_{i}\chi_{i}(x)=1$ and $\displaystyle \sum_{i}\nabla\chi_{i}(x)=0$  for all $x\in \Omega$. We have
\begin{align*}
\nabla g &= \nabla f -\sum_{i}\nabla b_{i}
\\
&=\nabla f-(\sum_{i}\chi_{i})\nabla f -\sum_{i}(f-f_{B_{i}})\nabla
\chi_{i}
\\
&=\ind_{F}(\nabla f) +\sum_{i}f_{B_{i}}\nabla \chi_{i}.
\end{align*}
From the definition of $F$ and the Lebesgue differentiation theorem, we have that $\ind_{F}(|f|+|\nabla f|)\leq \alpha\;\mu -$a.e.. We claim that a similar estimate holds for $h=\sum_{i}f_{B_{i}}\nabla \chi_{i}$. We have $|h(x)|\leq C\alpha$ for all $x\in M$. For this, note first that $h$ vanishes on $F$ and is locally finite on $\Omega$.
Then fix  $x\in \Omega$ and let $B_{j}$  be some Whitney ball containing $x$. For all $i\in I_{x}$, we have $|f_{B_{i}}-f_{B_{j}}|\leq Cr_{j} \alpha$.
Indeed, since $B_{i} \subset 7B_{j}$, we get
\begin{align}
|f_{B_{i}}-f_{7B_{j}}| &\leq
\frac{1}{\mu(B_{i})}\int_{B_{i}}|f-f_{7B_{j}}|d\mu\nonumber
\\
&\leq \frac{C}{\mu(B_{j})}\int_{7B_{j}}|f-f_{7B_{j}}|d\mu \nonumber
\\
&\leq Cr_{j}(\aver{7B_{j}}|\nabla f|^{q}d\mu)^{\frac{1}{q}} \nonumber 
\\
&\leq Cr_{j}\alpha  \label{g}
\end{align}
where we used  H\"{o}lder inequality, $(D)$, $(P_{q})$ and (\ref{f}). Analogously $|f_{7B_{j}}-f_{B_{j}}|\leq Cr_{j}\alpha$. Hence 
\begin{align*}
|h(x)| &=|\sum_{i\in I_{x}}(f_{B_{i}}-f_{B_{j}})\nabla \chi_{i}(x)|
\\
&\leq C\sum_{i\in I_{x}}|f_{B_{i}}-f_{B_{j}}|r_{i}^{-1}
\\
&\leq CN\alpha .
\end{align*}
From these estimates we deduce that $|\nabla g(x)|\leq C\alpha\; \mu- a.e.$. Let us now estimate $\| g \|_{\infty}$. We have $\displaystyle g=f\ind_{F}+\sum_{i}f_{B_{i}}\chi_{i}$. Since $|f|\ind_{F}\leq \alpha$, still need to estimate $\|\sum_{i}f_{B_{i}}\chi_{i}\|_{\infty}$. Note that
 \begin{align}
 |f_{B_{i}}|^{q}&\leq
  C\Bigl(\frac{1}{\mu(\overline{B_{i}})}\int_{\overline{B_{i}}}|f|d\mu \Bigr)^{q}\nonumber
\\
 &\leq \Bigl(\mathcal{M}(|f|+|\nabla f|)\Bigr)^{q}(y)\nonumber
 \\
 &\leq \mathcal{M}(|f|+|\nabla f|)^{q}(y)\nonumber
 \\
 &\leq \alpha^{q} \label{f_B}
 \end{align}
 where $y\in \overline{B_{i}}\cap F$ since $\overline{B_{i}}\cap F\neq \emptyset$. The second inequality follows from the fact that $(\mathcal{M}f)^{q}\leq \mathcal{M}f^{q}$ for $q\geq 1$.\\
Let $x\in \Omega$. Inequality (\ref{f_B}) and the fact that $\sharp I_{x}\leq N$ yield
 \begin{align*}
 |g(x)|&=|\sum_{i\in I_{x}}f_{B_{i}}|
 \\
 &\leq \sum_{i\in I_{x}}|f_{B_{i}}|
 \\
 &\leq N\alpha.
 \end{align*}
 We conclude that $\|g\|_{\infty} \leq C\,\alpha\quad \mu -a.e.$ and the proof of Proposition \ref{CZ} is therefore complete.
 \end{proof}
 \begin{rem} 1- It is a straightforward consequence of (\ref{eb}) that $b_{i}\in W_{r}^{1}$ for all $1\leq r\leq q$ with $\|b_{i}\|_{W_r^1}\leq C\alpha \mu(B_{i})^{\frac{1}{r}}$. \\
2- From the construction of the functions $b_{i}$, we see that $\sum_{i}b_{i}\in W_{p}^{1}$, with $\|\sum_{i}b_{i}\|_{W_{p}^{1}}\leq C \|f\|_{W_{p}^{1}}$.  It follows that $g\in W_{p}^{1}$. Hence $(g, |\nabla g|)$ satisfies the Poincar\'e inequality $(P_{p})$. Theorem 3.2 of \cite{hajlasz4} asserts that for $\mu-a.e. \; x,\,y \in M$
 $$
 |g(x)-g(y)|\leq  C d(x,y) \left((\mathcal{M}|\nabla g|^{p})^{\frac{1}{p}}(x)+(\mathcal{M}|\nabla g|^{p})^{\frac{1}{p}}(y)\right).
 $$
 From Theorem \ref{MIT} with $p=\infty$ and the inequality $\|\,|\nabla g|\,\|_{\infty}\leq C\alpha$, we deduce that $g$ has a Lipschitz representative. Moreover, the Lipschitz constant is controlled by $C\alpha$.
 \\
 3- We also deduce from this Calder\'on-Zygmund decomposition that $g\in W_{s}^{1}$ for $p\leq s\leq \infty$. We have  $\left(\int_{\Omega}(|g|^s+|\nabla g|^s)d\mu\right)^{\frac{1}{s}}\leq C\alpha \mu(\Omega)^{\frac{1}{s}}$ and 
 \begin{align*}\int_{F} (|g|^{s}+|\nabla g|^s) d\mu&=\int_{F}(|f|^{s}+|\nabla f|^{s})d\mu 
\\
&\leq \int_{F}(|f|^{p}|f|^{s-p}+|\nabla f|^{p}|\nabla f|^{s-p})d\mu \\
&\leq \alpha^{s-p} \|f\|_{W_{p}^{1}}^{p}<\infty.
\end{align*}
 \end{rem}
 
  %Let us prove the following proposition for later use, although it is not necessary for the proof of Theorem \ref{EK}.
\begin{cor}\label{CCZ}
Under the same hypotheses as in the Calder\'{o}n-Zygmund lemma, we have
$$
 W_{p}^{1}\subset W_{r}^{1}+W_{s}^{1} \quad \textrm{for}
\; 1\leq r\leq q\leq p\leq s<\infty.
$$
\end{cor}
\begin{proof}[Proof of Theorem \ref{EK}]
To prove part 1., we begin applying Theorem \ref{IK}, part 1. We have
$$
K(f,t^{\frac{1}{r}},L_{r},L_{\infty})\sim\Bigl(\int_{0}^{t}(f^{*}(s))^{r}ds\Bigr)^\frac{1}{r}.
$$
On the other hand
\begin{equation*}
\begin{aligned}
\Bigl(\int_{0}^{t}f^{*}(s)^{r}ds\Bigr)^\frac{1}{r}
&=\Bigl(\int_{0}^{t}|f(s)|^{r*}ds\Bigr)^\frac{1}{r}
\\& = \Bigl(t|f|^{r**}(t)\Bigr)^{\frac{1}{r}}
\end{aligned}
 \end{equation*}
where in the first equality we used the fact that
$f^{*r}=(|f|^{r})^{*}$ and the second follows from the definition of $f^{**}$. We thus get $ K(f,t^{\frac{1}{r}},L_{r},L_{\infty})\sim 
t^{\frac{1}{r}} (|f|^{r**})^{\frac{1}{r}}(t)$.
Moreover, 
$$
K(f,t^{\frac{1}{r}}, W_{r}^{1}, W_{\infty}^{1})\geq K(f,t^{\frac{1}{r}},L_{r},L_{\infty})+K(|\nabla f|,t^{\frac{1}{r}},L_{r},L_{\infty})
$$
since the linear operator
$$
(I,\,\nabla): W_{s}^{1}(M)\rightarrow (L_{s}(M;\mathbb{C}\times TM))
$$
is bounded for every $1\leq s\leq \infty$. These two points yield the desired inequality.

We will now prove part 2.. We treat the case when $f\in W_{p}^{1},\, q\leq p<\infty$.
Let $t>0$. We consider the Calder\'{o}n-Zygmund decomposition of $f$ of Proposition \ref{CZ} with
$\alpha=\alpha(t)=\Bigl(\mathcal{M}(|f|+|\nabla f|)^{q}\Bigr)^{*\frac{1}{q}}(t)$.
We write $ \displaystyle f=\sum_{i}b_{i}+g=b+g $ where
$(b_{i})_{i},\,g$ satisfy the properties of the proposition. From the bounded overlap property of the $B_{i}$'s, it follows that for all $r\leq q$
\begin{align*}
\| b \|_{r}^{r}&\leq \int_{M}(\sum_{i}
|b_{i}|)^{r}d\mu
\\
&\leq N \sum_{i}\int_{B_{i}}
|b_{i}|^{r}d\mu
\\
&\leq
C\alpha^{r}(t)\sum_{i}\mu(B_{i})
\\
&\leq C\alpha^{r}(t)\mu(\Omega).
\end{align*}
%This follows from the fact that
%$\displaystyle \sum_{i}\chi_{B_{i}}\leq N$ and
%$\Omega=\underset{i}{\bigcup}B_{i}$. Therefore $\|b\|_{q}\leq
%C\alpha(t)\mu(\Omega)^{\frac{1}{q}}\,$ and 
Similarly we have $\| \,|\nabla b|\,\|_{r}\leq
C\alpha(t)\mu(\Omega)^{\frac{1}{r}}$. 

Moreover, since $(\mathcal{M}f)^{*}\sim f^{**}$ and $(f+g)^{**}\leq f^{**}+g^{**}$, we get $$
\alpha(t)=\left(\mathcal{M}(|f|+|\nabla f|)^{q}\right)^{*\frac{1}{q}}(t)\leq C\left(|f|^{q**{\frac{1}{q}}}(t)+|\nabla
f|^{q**{\frac{1}{q}}}(t)\right).
$$
Noting that $\mu(\Omega)\leq t$, we deduce that
 \begin{equation}\label{Kr}
 K(f,t^{\frac{1}{r}},W_{r}^{1},W_{\infty}^{1})\leq
Ct^{\frac{1}{r}}\left(|f|^{q**\frac{1}{q}}(t)+|\nabla
f|^{q**\frac{1}{q}}(t)\right)
\end{equation}
for all $t>0$ and obtain the desired inequality for $f\in W_p^1,\, q\leq p<\infty$.

Note that in the special case where $r=q$, we have the upper bound of $K$ for $f\in W_q^1$. Applying a similar argument to that of \cite{devore1} --Euclidean case-- we get (\ref{Kr}) for $f\in W_q^1+W_\infty$. Here we will omit the details.
\end{proof}
 We were not able to show this characterization when $r<q$ since we could not show its validity even for $f\in W_r^1$. Nevertheless this theorem is enough to achieve interpolation (see the next section).

\subsubsection{The local case}
Let $M$ be a complete non-compact Riemannian manifold satisfying a local doubling property $(D_ {loc})$ 
%\footnote{Here we can not conclude like in the global case that $\mu(M)=\infty$.} 
 and a local Poincar\'{e} inequality $(P_{qloc})$ for some $1\leq q<\infty$.
 
Denote by $\mathcal{M}_{E}$ the Hardy-Littlewood maximal operator relative to a measurable subset 
$E$ of $M$, that is, for  $x\,\in E$ and every locally integrable function $f$ on $M$
$$
\displaystyle\mathcal{M}_{E}f(x)= \sup_{B:\,x\in
B}\frac{1}{\mu(B\cap E)}\int_{B \cap E}|f|d\mu
$$
where $B$ ranges over all open balls of $M$ containing $x$ and centered in $E$.
We say that a measurable subset $E$ of $M$ has the relative doubling property if there exists a constant $C_{E}$ such that for all $x\in E$ and $r>0$ we have
$$
\mu(B(x,2r)\cap E)\leq C_{E}\mu(B(x,r)\cap E).
$$
This is equivalent to saying that the metric-measure space $(E,d|_E,\mu|_E)$ has the doubling property.
On such a set $\mathcal{M}_{E}$ is of weak type $(1,1)$ and bounded on $L_{p}(E,\mu),\,1<p\leq\infty$.
\begin{proof}[Proof of Theorem \ref{EK}] To fix ideas, we assume without loss of generality $r_{0}=5$, $r_{1}=8$.  The lower bound of $K$ is trivial (same proof as for the global case). It remains to prove the upper bound. 
 
For all $t>0$, take $\alpha=\alpha(t)=\Bigl(\mathcal{M}(|f|+|\nabla f|)^{q}\Bigr)^{*\frac{1}{q}}(t)$.
Consider 
$$
\Omega=\left\lbrace x\in M:\mathcal{M}(|f|+|\nabla
f|)^{q}(x)>\alpha^{q}(t)\right\rbrace.
$$
We have $\mu(\Omega)\leq t$. If $\Omega=M$  then
\begin{align*}
 \int_{M}|f|^{r}d\mu+\int_{M}|\nabla f|^{r}d\mu
 &=\int_{\Omega}|f|^{r}d\mu+\int_{\Omega}|\nabla f|^{r}d\mu
\\
&\leq \int_{0}^{\mu(\Omega)}|f|^{r*}(s)ds+
\int_{0}^{\mu(\Omega)}|\nabla f|^{r*}(s)ds
\\
&\leq \int_{0}^{t}|f|^{r*}(s)ds+
\int_{0}^{t}|\nabla f|^{r*}(s)ds
\\
&=t\left(|f|^{r**}(t)+|\nabla f|^{r**}(t)\right).
\end{align*}
Therefore
\begin{align*}
 K(f,t^{\frac{1}{q}},W_{r}^{1},W_{\infty}^{1})
 &\leq \| f \|_{W_{r}^{1}}
\\
&\leq Ct^{\frac{1}{r}}\left(|f|^{r**\frac{1}{r}}(t)+|\nabla f|^{r**\frac{1}{r}}(t)\right)
\\
&\leq Ct^{\frac{1}{r}}\Bigl(|f|^{q**\frac{1}{q}}(t)+|\nabla f|^{q**\frac{1}{q}}(t)\Bigr)
\end{align*}
since $r\leq q$. We thus obtain the upper bound in this case. 

Now assume $\Omega \neq M$. Pick a countable set
$\left\lbrace x_{j}\right\rbrace _{j\in J} \subset  M,$ such that $ M=
\underset{j\in J}{\bigcup}B(x_{j},\frac{1}{2})$ and for all $x\in M$,
$x$ does not belong to more than $N_{1}$ balls $B^{j}:=B(x_{j},1)$.
Consider a $C^{\infty}$ partition of unity $(\varphi_{j})_{j\in J}$ subordinated to the balls $\frac{1}{2}B^{j}$ such that $0\leq
\varphi_{j}\leq 1,\,\supp\varphi_{j}\subset B^{j}$  and
$\|\,|\nabla \varphi_{j}|\, \|_{\infty}\leq C$ uniformly with respect to $j$.
Consider $f\in W_{p}^{1}$, $ q\leq p<\infty$.
 Let $ f_{j}=f\varphi_{j}$ so that $ f=\sum_{j\in J}f_{j}$. We have for $j\in J$, $f_{j}\in L_{p}$ and $\;\nabla f_{j}=f\nabla
\varphi_{j}+\nabla f \varphi_{j} \in L_{p}$. Hence $f_{j}\in W_{p}^{1}(B^{j})$.
The balls $B^{j}$ satisfy  the relative doubling property with constant independent of the balls $B^{j}$. This follows from the next lemma quoted from \cite{auscher2} p.947.

\begin{lem}\label{DB}
Let $M$ be a  complete Riemannian manifold satisfying
$(D_ {loc})$. Then the balls $B^{j}$ above, 
equipped with the induced distance and measure, satisfy the
relative doubling property $(D)$, with the doubling constant that may be chosen independently of $j$.
More precisely, there exists $C\geq0$ such that for all $j\in J$
\begin{equation}
 \mu(B(x,2r)\cap B^{j})\leq
C\,\mu(B(x,r)\cap B^{j})\label{DB1}
\quad\forall x\in B^{j},\,r>0,
\end{equation}
and
\begin{equation}
\mu(B(x,r))\leq C\mu(B(x,r)\cap B^{j})\label{DB2}
\quad \forall x\in B^{j},\,0<r\leq 2.
 \end{equation}
\end{lem}

\begin{rem}\label{RB}
Noting that the proof in \cite{auscher2} only used the fact that $M$ is a length space, we observe that Lemma \ref{DB} still holds for any length space. Recall that a length space $X$ is a metric space such that the distance between any two points $x,\,y \in X $ is equal to the infimum of the lengths of all paths joining $x$ to $y$ (we implicitly assume that there is at least one such path). Here a path from $x$ to $y$ is a continuous map $\gamma:[0,1]\rightarrow X$ with $\gamma(0)=x$ and $\gamma(1)=y$. 
\end{rem} 
\noindent Let us return to the proof of the theorem.
For any $x\in B^{j}$ we have
\begin{align}
 \mathcal{M}_{B^{j}}(|f_{j}|+|\nabla f_{j}|)^{q}(x)
 &=\sup_{B:\,x\in B,\,r(B)\leq 2}\frac{1}{\mu(B^{j}\cap
B)}\int_{B^{j}\cap B}(|f_{j}|+|\nabla f_{j}|)^{q}d\mu \nonumber
\\
&\leq \sup_{B:\,x\,\in B,\;r(B)\leq 2}C\frac{\mu(B)}{\mu(B^{j}\cap B)}\frac{1}{\mu(B)}\int_{B}(|f|
+|\nabla f|)^{q}d\mu \nonumber
\\
&\leq C\mathcal{M}(|f|+|\nabla f|)^{q}(x)\label{MB}
\end{align}
where we used (\ref{DB2}) of Lemma \ref{DB}. Consider now 
$$
\Omega_{j}=\left\lbrace x\in
B^{j}:\mathcal{M}_{B^{j}}(|f_{j}|+|\nabla
f_{j}|)^{q}(x)>C\alpha^{q}(t)\right\rbrace
$$
 where $C$ is the constant in
(\ref{MB}). $\Omega_{j}$ is an open subset of $B^{j}$, hence of $M$, and $\Omega_{j}\subset \Omega\neq M $ for all $j \in J$. 
For the $f_{j}$'s, and for all $t>0$, we have a Calder\'{o}n-Zygmund decomposition similar to the one done in Proposition \ref{CZ}:
there exist $b_{jk},\;g_{j}$ supported in $B^{j}$, and balls $(B_{jk})_{k}$ of $M$, contained in $\Omega_{j}$, such that
\begin{equation}
f_{j} = g_{j}+\sum_{k}b_{jk} \label{f1}
\end{equation}
\begin{equation}
|g_{j}(x)|\leq C\alpha(t) \textrm{   and } \,|\nabla g_{j}(x)|\leq
C\alpha(t) \quad \textrm{for } \mu -a.e. \, x \in M \label{eg1}
\end{equation}
\begin{equation}
\supp b_{jk}\subset B_{jk}, \textrm { for } 1\leq r\leq q\, \int_{B_{jk}}(|b_{jk}|^{r}+ |\nabla b_{jk}|^{r})d\mu
\leq C\alpha^{r}(t)\mu(B_{jk})\label{eb1}
\end{equation}
\begin{equation}
\sum_{k}\mu(B_{jk})\leq C\alpha^{-p}(t)\int_{B^{j}} (|f_{j}|+|\nabla
f_{j}|)^{p}d\mu \label{emB1}
\end{equation}
\begin{equation}
\sum_{k}\chi_{B_{jk}}\leq N \label{eiB1}
\end{equation}
with $C$ and $N$ depending only on $q$, $p$ and the constants in $(D_{loc})$ and $(P_{qloc})$. The proof of this decomposition will be the same as in Proposition \ref{CZ}, taking for all $j\in J$ a Whitney decomposition  $(B_{jk})_{k}$ of $\Omega_{j}\neq M$ and using the doubling property for balls whose radii do not exceed $3<r_{0}$ and the Poincar\'{e} inequality for balls whose radii do not exceed $7<r_{1}$. For the bounded overlap property (\ref{eiB1}), just note that the radius of every ball $B_{jk}$ is less than 1. Then apply the same argument as for the bounded overlap property of a Whitney decomposition for an homogeneous space, using the doubling property for balls with sufficiently small radii.  

By the above decomposition we can write $ f=\sum\limits _{j\in J}\sum\limits_{k}b_{jk}+\sum\limits _{j\in J}g_{j}=b+g$.
Let us now estimate $\|b\|_{W_{r}^{1}}$ and $\|g\|_{W_{\infty}^{1}}$. 
\begin{align*} 
\|b\|_{r}^{r}&\leq N_{1}N\sum_{j}
\sum_{k}\|b_{jk}\|_{r}^{r}
\\
&\leq C\alpha^{r}(t) \sum_{j}\sum_{k}(\mu(B_{jk}))
\\
&\leq NC\alpha^{r}(t) \Bigl(\sum_{j}\mu(\Omega_{j})\Bigr)
\\
&\leq N_{1}C\alpha^{r}(t) \mu(\Omega).
\end{align*}
We used the bounded overlap property of the $(\Omega_{j})_{j\in J}$'s and that of the 
$(B_{jk})_{k}$'s for all $j\in J$. 
It follows that $\| b\|_{r}\leq C\alpha(t)\mu(\Omega)^{\frac{1}{r}}$.
Similarly we get $\|\,|\nabla b|\, \|_{r}\leq C\alpha(t)\mu(\Omega)^{\frac{1}{r}}$.

For $g$ we have
\begin{align*}
\|g \|_{\infty}& \leq \sup_{x}\sum_{j \in J}|g_{j}(x)|
\\
& \leq \sup_{x}N_{1}\sup_{j\in J}|g_{j}(x)|
\\
&\leq N_{1}\sup_{j\in J}\|
g_{j}\|_{\infty}
\\
&\leq C\alpha(t).
\end{align*}
Analogously $\| \,|\nabla g|\,\|_{\infty}\leq
C\alpha(t)$. We conclude that
\begin{align*}
 K(f,t^{\frac{1}{r}},W_{r}^{1},W_{\infty}^{1})
&\leq \| b \|_{W_{r}^{1}}+t^{\frac{1}{r}}\|g\|_{W_{\infty}^{1}}
\\
&\leq C
\alpha(t)\mu(\Omega)^{\frac{1}{r}}+Ct^{\frac{1}{r}}\alpha(t)
\\
&\leq Ct^{\frac{1}{r}}\alpha(t)
\\
&\sim Ct^{\frac{1}{r}}(|f|^{q**\frac{1}{q}}(t)+|\nabla
f|^{q**\frac{1}{q}}(t))
\end{align*}
which completes the proof of Theorem \ref{EK} in the case $r<q$. When $r=q$ we get the characterization of $K$ for every $f\in W_q^1+W_\infty^1$ by applying again a similar argument to that of \cite{devore1}.
\end{proof}
\
\section{Interpolation Theorems}\label{TI} In this section we establish our interpolation Theorem \ref{IS} and some  consequences for non-homogeneous Sobolev spaces on a complete non-compact Riemannian manifold $M$ satisfying $(D_{loc})$ and $(P_{qloc})$ for some $1\leq q< \infty$.

For $1\leq r\leq q<p<\infty$, we define the interpolation space $W_{p,r}^{1}$ between $ W_{r}^{1}$ and $ W_{\infty}^{1}$ by
$$
W_{p,r}^{1}=(W_{r}^{1},W_{\infty}^{1})_{1-\frac{r}{p},p}.
$$
From the previous results we know that 
$$
C_{1}\left\lbrace\int_{0}^{\infty}\left(t^{\frac{1}{p}}(|f|^{r**\frac{1}{r}}+|\nabla f|^{r**\frac{1}{r}})(t)\right)^{p}\frac{dt}{t}\right\rbrace^{\frac{1}{p}}\leq\|f\|_{1-\frac{r}{p},p}\leq C_{2}\left\lbrace\int_{0}^{\infty}\left(t^{\frac{1}{p}}(|f|^{q**\frac{1}{q}}+|\nabla f|^{q**\frac{1}{q}})(t)\right)^{p}\frac{dt}{t}\right\rbrace^{\frac{1}{p}}.
$$

 We claim that $W_{p,r}^{1}= W_{p}^{1}$, with equivalent norms. Indeed,
 \begin{align*}
\|f\|_{1-\frac{r}{p},p}&\geq C_{1} \left\lbrace\int_{0}^{\infty}\left(|f|^{r**\frac{1}{r}(t)}+|\nabla f|^{r**\frac{1}{r}}(t)\right)^{p}dt\right\rbrace^{\frac{1}{p}}
\\
&\geq C\left( \|f^{r**}\|_{\frac{p}{r}}^{\frac{1}{r}}
+\| |\nabla f|^{r**} \|_{\frac{p}{r}}^{\frac{1}{r}}\right)
\\
&\geq C\left(\|f^{r}\|_{\frac{p}{r}}^{\frac{1}{r}}
+\| \,|\nabla f|^{r}\,\|_{\frac{p}{r}}^{\frac{1}{r}}\right)
\\
&=C\left(\|f\|_{p} + \| \,|\nabla f|\, \|_{p}\right)
\\
&=C \| f\|_{W_{p}^{1}},
\end{align*}
and 
\begin{align*}
\|f\|_{1-\frac{r}{p},p}&\leq C_{2} \left\lbrace\int_{0}^{\infty}\left(|f|^{q**\frac{1}{q}}(t)+|\nabla f|^{q**\frac{1}{q}}(t)\right)^{p}dt\right\rbrace^{\frac{1}{p}}
\\
&\leq C\left( \|f^{q**}\|_{\frac{p}{q}}^{\frac{1}{q}}
+\|\, |\nabla f|^{q**}\, \|_{\frac{p}{q}}^{\frac{1}{q}}\right)
\\
&\leq C\left(\|f^{q}\|_{\frac{p}{q}}^{\frac{1}{q}}
+\|\, |\nabla f|^{q}\,\|_{\frac{p}{q}}^{\frac{1}{q}}\right)
\\
&=C\left(\|f\|_{p} + \|\, |\nabla f| \,\|_{p}\right)
\\
&=C \| f\|_{W_{p}^{1}},
\end{align*}
where we used that for $l>1$, $\|f^{**}\|_{l}\sim \|f\|_{l}$ (see \cite{stein3}, Chapter V: Lemma 3.21 p.191 and Theorem 3.21, p.201). Moreover, from Corollary \ref{CCZ}, we have $W_{p}^{1}\subset
 W_{r}^{1}+W_{\infty}^{1}$ for $ r<p<\infty$. Therefore $W_{p}^{1}$ is an interpolation space between $W_{r}^{1}$
 and  $W_{\infty}^{1}$ for $\,r<p<\infty$.
 
Let us recall some known facts about Poincar\'{e} inequalities with varying $q$.
 \\
It is known that $(P_{qloc})$ implies $(P_{ploc})$ when $p\geq q$ (see \cite{hajlasz4}). Thus if the set of $q$ such that $(P_{qloc})$ holds is not empty, then it is an interval unbounded on the right. A recent result of Keith and Zhong \cite{keith3} asserts that this interval is open in $[1,+\infty[$.
\begin{thm}\label{kz} Let $(X,d,\mu)$ be a complete metric-measure space with $\mu$ locally doubling
and admitting a local Poincar\'{e} inequality $(P_{qloc})$, for  some $1< q<\infty$.
Then there exists $\epsilon >0$ such that $(X,d,\mu)$ admits
$(P_{ploc})$ for every $p>q-\epsilon$. 
\end{thm}
Here, the definition of $(P_{qloc})$ is that of section 7. It reduces to the one of section 3 when the metric space is a Riemannian manifold.
\begin{proof}[Comment on the proof of this theorem] The proof goes as in \cite{keith3} where this theorem is proved for $X$ satisfying $(D)$ and admitting a global Poincar\'{e} inequality $(P_{q})$. By using the same argument and choosing sufficiently small radii for the considered balls, $(P_{qloc})$ will give us $(P_{(q-\epsilon)loc})$ for every ball of radius less than $r_{2}$, for some $r_{2}<\min(r_{0},r_{1})$, $r_{0},\,r_{1}$ being the constants given in the definitions of local doubling property and local Poincar\'{e} inequality.
\end{proof}
 Define $A_{M}=\left\lbrace q \in [1,\infty[:(P_{qloc})\textrm{ holds }\right\rbrace$ and $q_{0_{M}}=\inf A_{M}$. When no confusion arises, we write $q_{0}$ instead of $q_{0_{M}}$.
 As we mentioned in the introduction, this improvement of the exponent of a Poincar\'e inequality together with the reiteration theorem yield another version of our interpolation result: Corollary \ref{CIS}.
\begin{proof}[Proof of Corollary \ref{CIS}]  Let $0<\theta<1$ such that $\frac{1}{p}=\frac{1-\theta}{p_1}+\frac{\theta}{p_{2}}$.
 \begin{itemize}
\item[1.] \textit{Case when $p_{1}>q_{0}$}. Since $p_{1}>q_{0}$, there exists $q\in A_{M}$ such that $q_{0}<q< p_{1}$. Then $1-\frac{q}{p}=(1-\theta)(1-\frac{q}{p_{1}}) +\theta (1-\frac{q}{p_{2}})$.
 The reiteration theorem --\cite{bennett}, Theorem 2.4 p.110-- yields 
 \begin{align*}
(W_{p_{1}}^{1}  ,W_{p_{2}}^{1})_{\theta,p}&=(W_{p_{1},q}^{1} ,W_{p_{2},q}^{1})_{\theta,p}\\
&= (W_{q}^{1} ,W_{\infty}^{1})_{1-\frac{q}{p},p}
 \\
 &=W_{p,q}^{1}
 \\
 &= W_{p}^{1}.
\end{align*}
\item[2.]{Case when $1\leq p_{1}\leq q_{0}$}. Let $\theta'=\theta (1-\frac{p_1}{p_{2}})=1-\frac{p_1}{p}$. The reiteration theorem applied this time only to the second exponent yields
\begin{align*}
(W_{p_{1}}^{1}  ,W_{p_{2}}^{1})_{\theta,p}&=(W_{p_{1}}^{1}  ,W_{p_{2},p_{1}}^{1})_{\theta,p}\\
&= (W_{p_1}^{1} ,W_{\infty}^{1})_{\theta',p}
 \\
 &=W_{p,p_1}^{1}
 \\
 &= W_{p}^{1}.
 \end{align*}
 \end{itemize}
\end{proof}
% \begin{rem} If $(P_{1loc})$ holds (then $q_{0}=1$) we allow then $1\leq p_{1}$ in Corollary \ref{reitS}.
 % \end{rem}
 \begin{thm} Let $M$ and $N$ be two complete non-compact Riemannian manifolds satisfying
 $(D_{loc})$. Assume that $q_{0_{M}}$ and $q_{0_{N}}$ are well defined. Take $1\leq p_{1}\leq p_{2}\leq \infty,\,
 1\leq  r_{1},\,r,\, r_{2}\leq \infty$.
 Let $T$ be a bounded linear operator from
  $ W_{p_{i}}^{1}(M)$ to $ W_{r_{i}}^{1}(N)$ of norm $L_{i},\;
  i=1,2$. Then for every couple $(p,r)$  such that $p\leq r$, $p>q_{0_{M}}$, $r>q_{0_{N}}$  and $ (\frac{1}{p},\frac{1}{r})= (1-\theta)(\frac{1}{p_{1}},\frac{1}{r_{1}})+
  \theta  (\frac{1}{p_{2}},\frac{1}{r_{2}})$, $0<\theta<1$, $T$ is  bounded from $ W_{p}^{1}(M)$ to $
W_{r}^{1}(N)$ with norm $L\leq
CL_{0}^{1-\theta}L_{1}^{\theta}$.
 \end{thm}
 \begin{proof}
  \begin{align*}
  \|Tf\|_{W_{r}^{1}(N)} &\leq C \|Tf\|_{(W_{r_{1}}^{1}(N),W_{r_{2}}^{1}(N))_{\theta,r}}
 \\
 &\leq CL_{0}^{1-\theta}L_{1}^{\theta}\|f\|_{(W_{p_{1}}^{1}(M),W_{p_{2}}^{1}(M))_{\theta,r}}
 \\
 &\leq CL_{0}^{1-\theta}L_{1}^{\theta}\|f\|_{(W_{p_{1}}^{1}(M),W_{p_{2}}^{1}(M))_{\theta,p}}
 \\
 &\leq CL_{0}^{1-\theta}L_{1}^{\theta}\|f\|_{W_{p}^{1}(M)}.
 \end{align*}
We used the fact that $K_{\theta,q}$ is an exact interpolation functor of exponent $\theta$, that $W_{p}^{1}(M)=(W_{p_{1}}^{1}(M),W_{p_{2}}^{1}(M))_{\theta,p}$, $W_{r}^{1}(N)=(W_{r_{1}}^{1}(N),W_{r_{2}}^{1}(N))_{\theta,r}$ with equivalent norms and that $(W_{p_{1}}^{1}(M),W_{p_{2}}^{1}(M))_{\theta,p}\subset (W_{p_{1}}^{1}(M),W_{p_{2}}^{1}(M))_{\theta,r}$ if $p\leq r$.
\end{proof}
\begin{rem}\label{C} Let $M$ be a Riemannian manifold, not necessarily complete, satisfying $(D_{loc})$. Assume that for some $1\leq q<\infty$, a weak local Poincar\'{e} inequality holds for all $C^{\infty}$ functions, that is there exists $r_{1}>0,\, C=C(q,r_{1}),\,\lambda \geq 1$ such that for all $f\in C^{\infty}$ and all ball $B$ of radius $r<r_{1}$ we have
 \begin{equation*} 
\Bigl(\aver{B}|f-f_{B}|^{q}d\mu\Bigr)^{\frac{1}{q}}\leq Cr \Bigl(\aver{\lambda B}|\nabla f|^{q}d\mu\Bigr)^{\frac{1}{q}}. 
\end{equation*}
 Then, we obtain the characterization of $K$ as in Theorem \ref{EK} and we get by interpolating a result analogous to Theorem \ref{IS}. 
 \end{rem}
  \section{Homogeneous Sobolev spaces on Riemannian manifolds}
\begin{dfn} Let $M$ be a $C^{\infty}$ Riemannian manifold of dimension $n$.
For $1\leq p\leq \infty$, we define $\overset{.}{E_{p}^{1}}$ to be the vector space of distributions $ \varphi $
 with $|\nabla \varphi |\in L_{p}$, where $\nabla \varphi$ is the distributional gradient of $ \varphi$. It is well known that the elements of $\overset{.}{E_{p}^{1}}$ are in  $L_{ploc} $. We equip $\overset{.}{E_{p}^{1}}$  with the semi norm 
 $$ 
 \|\varphi\|_{\overset{.}{E_{p}^{1}}}=\|\,|\nabla \varphi|\,\|_{p}.
 $$
 \end{dfn}
 \begin{dfn} We define the homogeneous Sobolev space $\overset{.}{W_{p}^{1}}$ as the quotient space $\overset{.}{E_{p}^{1}}/\mathbb{R}$.
 % for the norm
%$$ 
%\|\ \phi \|_{\overset{.}{W_{p}^{1}}}=\inf\left\lbrace \|\,|\nabla \varphi|\,\|_{p}, \, \varphi \in \overset{.}{E_{p}^{1}},\,\overline{\varphi}=\phi \right\rbrace
%$$
%where $\overline{\varphi}$ denotes the class of $\varphi$.
\end{dfn}
\begin{rem} For all $\varphi \in \overset{.}{E_{p}^{1}}$, \, $\|\overline{\varphi}\|_{\overset{.}{W_{p}^{1}}}=\|\,|\nabla \varphi|\,\|_{p}$, where $\overline{\varphi}$ denotes the class of $\varphi$.
\end{rem}
\begin{prop}\label{DBL}
%(\cite{hajlasz5})Lip\cap\overset{.}{W_{p}^{1}}$ is dense in $\overset{.}{W_{p}^{1}}$ for $1\leq p<\infty$.
1. (\cite{goldshtein}) $\overset{.}{W_{p}^{1}}$ is a Banach space.\\
2. Assume that $M$ satisfies $(D)$ and $(P_{q})$ for some $1\leq q<\infty$ and for all $f\in Lip$, that is there exists a constant $C>0$ such that for all $f\in Lip$ and for every ball $B$ of $M$ of radius $r>0$ we have
\begin{equation*} \tag{$P_{q}$}
\left(\aver{B}|f-f_{B}|^{q} d\mu\right)^{\frac{1}{q}} \leq C r \left(\aver{B}|\nabla f|^{q}d\mu\right)^{\frac{1}{q}}.
\end{equation*}            
Then $Lip(M)\cap\overset{.}{W_{p}^{1}}$ is dense in $\overset{.}{W_{p}^{1}}$ for $q\leq p<\infty$.
  %           3. Let $1\leq p<\infty$ and $f\in\overset{.}{E_{p}^{1}}$. Then, to every $\epsilon>0$ there exists $h\in %C^{\infty}$ such that $\|f-h\|_{W_{p}^{1}}<\epsilon$. 
 \end{prop}      
\begin{proof} The proof of item 2. is implicit in the proof of Theorem 9 in \cite{franchi1}.
\end{proof}
We obtain for the $K$-functional of the homogeneous Sobolev spaces the following homogeneous  form of Theorem \ref{EK},  weaker in the particular case $r=q$, but again  sufficient for us to interpolate. 
\begin{thm}\label{EKH} Let $M$ be a complete Riemannian manifold satisfying $(D)$ and $(P_{q})$ for some $1\leq q<\infty$. Let $1\leq r\leq q$.  Then
\begin{itemize}
\item[1.] there exists $C_{1}$ such that for every $F \in\overset{.}{W_{r}^{1}}+\overset{.}{W_{\infty}^{1}}$ and all $t>0$
$$ K(F,t^\frac{1}{r},\overset{.}{W_{r}^{1}},\overset{.}{W_{\infty}^{1}})\geq C_{1}t^{\frac{1}{r}}|\nabla f|^{r**\frac{1}{r}}(t) \textrm{ where } \, f\in \overset{.}{E_{r}^{1}}+\overset{.}{E_{\infty}^{1}} \textrm{ and } \overline{f}=F;
$$
\item[2.] for $ q\leq p<\infty$, there exists $C_{2}$ such that for every $F\in \overset{.}{W_{p}^{1}}$  and every $t>0$
$$ 
K(F,t^{\frac{1}{r}},\overset{.}{W_{r}^{1}},\overset{.}{W_{\infty}^{1}})\leq C_{2} t^{\frac{1}{r}}|\nabla f|^{q**\frac{1}{q}}(t)\textrm{ where } \, f\in \overset{.}{E_{p}^{1}} \textrm{ and } \overline{f}=F.
$$
\end{itemize}

\end{thm}
Before we prove Theorem \ref{EKH}, we give the following Calder\'on-Zygmund decomposition that will be also in this case our principal tool to estimate $K$. 
\begin{prop}[Calder\'{o}n-Zygmund lemma for Sobolev functions]\label{CZH} Let $M$ be a complete non-compact Riemannian manifold satisfying $(D)$ and $(P_{q})$ for some $1\leq q<\infty$. 
Let $ q \leq p < \infty$, $ f \in \overset{.}{E_{p}^{1}}$ and $\alpha
>0$. Then there is a collection of balls $(B_{i})_{i}$, 
functions $b_{i}\in \overset{.}{E_{q}^{1}}$ and a Lipschitz function $g$ 
such that the following properties hold :
\begin{equation}
f = g+\sum_{i}b_{i} \label{dfn}
\end{equation}
\begin{equation}
|\nabla g(x)|\leq C\,\alpha
\quad \mu-a.e. \label{egn}
\end{equation}
\begin{equation}
\supp b_{i}\subset B_{i}\;\textrm {and for }  1\leq r\leq q\;
 \int_{B_{i}}|\nabla b_{i}|^{r}
d\mu \leq C\alpha^{r}\mu(B_{i}) \;\label{ebn}
\end{equation}
\begin{equation}
\sum_{i}\mu(B_{i})\leq C\alpha^{-p}\int |\nabla f|^{p}d\mu
\label{eBn}
\end{equation}
\begin{equation}
\sum_{i}\chi_{B_{i}}\leq N \label{rbn}.
\end{equation}
The constants $C$ and $N$ depend only on $q$, $p$ and the constant in
$(D)$.
\end{prop}
\begin{proof}The proof goes as in the case of non-homogeneous Sobolev spaces, but taking 
$\Omega=\left\lbrace x\in M:\mathcal{M}(|\nabla f|^{q})(x)>\alpha^{q}\right\rbrace$ as $\|f\|_{p}$ is not under control. We note that in the non-homogeneous case, we used that $f\in L_{p}$ only to control $g\in L_{\infty}$ and $b\in L_{r}$.
\end{proof}
\begin{rem}
It is sufficient for us that the Poincar\'{e} inequality holds for all $f\in \overset{.}{E_{p}^{1}}$.
\end{rem}
%\begin{prop}\label{glh} The function $g$ of Calder\'{o}n-Zygmund lemma is Lipschitz almost everywhere on $M$ and $|g(x)-g(y)|\leq C\alpha d(x,y)$ almost everywhere.
%\end{prop} 
%\begin{proof} Same proof as that of Proposition \ref{gl}.
%\end{proof}
\begin{cor}\label{CCZH}
Under the same hypotheses as in the Calder\'{o}n-Zygmund lemma, we have
$$ \overset{.}{W_{p}^{1}}\subset \overset{.}{W_{r}^{1}}+\overset{.}{W_{\infty}^{1}} \; \textrm{ for }
 1\leq r\leq q\leq p<\infty \,.
$$
\end{cor}
\begin{proof}[Proof of Theorem \ref{EKH}]
The proof of item 1. is the same as in the non-homogeneous case. Let us turn to inequality 2.. For $F\in\overset{.}{W_{p}^{1}}$ we take $f\in\overset{.}{E_{p}^{1}}$ with $\overline{f}=F$. Let $t>0$ and $\alpha(t)=\Bigl(\mathcal {M}(|\nabla f|^{q})\Bigr)^{*\frac{1}{q}}(t)$. By the Calder\'{o}n-Zygmund decomposition with $\alpha=\alpha(t)$, $f$ can be written $f=b+g$, hence $F=\overline{b}+\overline{g}$, with $\|\overline{b}\|_{\overset{.}{W_{r}^{1}}}=\|\,|\nabla b|\,\|_{r}\leq C\alpha(t) \mu(\Omega)^{\frac{1}{r}}$ and $\|\overline{g}\|_{\overset{.}{W_{\infty}^{1}}}=\||\,\nabla g|\,\|_{\infty}\leq C\alpha(t)$. Since for $\alpha=\alpha(t)\,$ we have $\mu(\Omega)\leq t$, then we get  $K(F,t^{\frac{1}{r}},\overset{.}{W_{r}^{1}},\overset{.}{W_{\infty}^{1}})\leq Ct^{\frac{1}{r}}|\nabla f|^{q**\frac{1}{q}}(t)$.
\end{proof} 
We can now prove our interpolation result for the homogeneous Sobolev spaces.
\begin{proof}[Proof of Theorem \ref{IHS}] The proof follows directly from Theorem \ref{EKH}. Indeed, item 1. of Theorem \ref{EKH} yields  
$$(\overset{.}{W_{r}^{1}},\overset{.}{W_{\infty}^{1}})_{1-\frac{r}{p},p} \subset \overset{.}{W_{p}^{1}}$$ 
with 
$\|F\|_{\overset{.}{W_{p}^{1}}}\leq C\|F\|_{1-\frac{r}{p},p}$, while  item 2. gives us that 
$$
\overset{.}{W_{p}^{1}}\subset (\overset{.}{W_{r}^{1}},\overset{.}{W_{\infty}^{1}})_{1-\frac{r}{p},p}
$$ 
with $\|F\|_{1-\frac{r}{p},p}\leq C\|F\|_{\overset{.}{W_{p}^{1}}}$. We conclude that
$$\overset{.}{W_{p}^{1}}= (\overset{.}{W_{r}^{1}},\overset{.}{W_{\infty}^{1}})_{1-\frac{r}{p},p}
$$
 with equivalent norms.
 \end{proof}
 
 \begin{cor}[The reiteration theorem]\label{RH} Let $M$ be a complete non-compact Riemannian manifold satisfying $(D)$ and $(P_{q})$ for some $1\leq q<\infty$. Define $q_{0}=\inf\left\lbrace q \in [1,\infty[:(P_{q})\textrm{ holds }\right\rbrace$. Then for $p>q_{0}$ and $1\leq p_{1}<p<p_{2}\leq \infty$,  $ \overset{.}{W_{p}^{1}}$ is an interpolation space between $\overset{.}{W_{p_{1}}^{1}}$ and $\overset{.}{W_{p_{2}}^{1}}$.
 \end{cor}
 
\paragraph{\textbf{Application}}

Consider a complete non-compact Riemannian manifold $M$ satisfying $(D)$ and $(P_{q})$ for some $1\leq q<2$.
Let $\Delta$ be the Laplace-Beltrami operator. Consider the linear operator  $\Delta^{\frac{1}{2}}$ with the following resolution
$$
\Delta^{\frac{1}{2}}f= c \int_{0}^{\infty}\Delta e^{-t\Delta}f\frac{dt}{\sqrt{t}} ,\quad  f\in C^{\infty}_{0}
$$
where $c=\pi^{-\frac{1}{2}}$. Here
$\Delta^{\frac{1}{2}}f$ can be defined for $f\in \Lip$ as a measurable function (see \cite{auscher1}).

In \cite{auscher1}, Auscher and Coulhon proved that on such a manifold, we have
$$
\mu\left\lbrace x\in M:|\Delta^{\frac{1}{2}}f(x)|>\alpha\right\rbrace\leq \frac{C}{\alpha^{q}}\|\,|\nabla f|\,\|_{q}
$$
for $f\in C_{0}^{\infty}$, with $q\in[1,2[$. In fact one can check that the argument applies to all $f\in \Lip\cap\overset{.}{E_{q}^{1}}$ and since $\Delta^{\frac{1}{2}}1=0$, $\Delta^{\frac{1}{2}}$ can be defined on $\Lip\cap\overset{.}{W_{q}^{1}}$ by taking quotient which we keep calling $\Delta^{\frac{1}{2}}$. Moreover, Proposition \ref{DBL} gives us that $\Delta^{\frac{1}{2}}$ has a bounded extension from $\overset{.}{W_{q}^{1}}$ to $L_{q,\infty}$. 
Since we already have 
\begin{equation*}
 \| \Delta^{\frac{1}{2}}f\|_{2}\leq  \|\,|\nabla f|\,\|_{2}
\end{equation*}
then by Corollary \ref{RH}, we see at once 
\begin{equation}\label{RRp}
 \| \Delta^{\frac{1}{2}}f\|_{p}\leq C_{p} \|\,|\nabla f|\,\|_{p}
 \end{equation}
 for all $q<p\leq 2$ and $f\in \overset{.}{W_{p}^{1}}$, without using the argument in \cite{auscher1}. 

 \section{Sobolev spaces on compact manifolds}
Let $M$ be a $C^{\infty}$ compact manifold equipped with a Riemannian metric. Then $M$ satisfies then the doubling property $(D)$ and  the Poincar\'{e} inequality $(P_{1})$.
\begin{thm} Let $M$ be a $C^{\infty}$ compact Riemannian manifold. There exist $C_{1},\,C_{2}\;$ such that for all  $f \in
 W^{1}_{1}+W^{1}_{\infty}$ and all $t>0$ we have
 \begin{equation*} \tag{$\ast_{\textrm{comp}}$}
C_{1} t\Bigl(|f|^{ **}(t)+|\nabla
f|^{**}(t)\Bigr)\leq
K(f,t,W^{1}_{1},W^{1}_{\infty})\leq C_{2}t
\Bigl(|f|^{**}(t)+|\nabla
f|^{**}(t)\Bigr).
\end{equation*}
 \end{thm}
\begin{proof}
It remains to prove the upper bound for $K$ as the lower bound is trivial. Indeed, let us consider for all $t>0$ and for $\alpha(t)=\left(\mathcal{M}(|f|+|\nabla f|)\right)^{*}(t)$, $\Omega=\left\lbrace x\in M;\mathcal{M}(|f|+|\nabla f|)(x)\geq \alpha(t)\right\rbrace$. If $\Omega\neq M$, we have the Calder\'{o}n-Zygmund decomposition as in Proposition \ref{CZ} with $q=1$ and the proof will be the same as the proof of Theorem \ref{EK} in the global case.
Now if $\Omega=M$, we prove the upper bound  by the same argument used in the proof of Theorem \ref{EK} in the local case.
Thus, in the two cases we obtain the right hand inequality of $(\ast_{\textrm{comp}})$ for all $f\in W_{1}^{1}+W_{\infty}^{1}$. 
\end{proof}
It follows that
\begin{thm} For all $\, 1\leq p_{1}<p<p_{2}\leq \infty$, $W_{p}^{1}$ is an interpolation space between $W_{p_{1}}^{1}$ and $W_{p_{2}}^{1}$. 
\end{thm}
\section{Metric-measure spaces}
In this section we consider $(X,d,\mu)$ a metric-measure space with $\mu$ doubling. 
\subsection{Upper gradients and Poincar\'{e} inequality}
\begin{dfn}[Upper gradient \cite{heinonen5}]
Let $u:X\rightarrow \mathbb{R}$ be a Borel function. We say that a Borel function $g:X\rightarrow [0,+\infty]$  is an  upper gradient of $u$ if $|u(\gamma(b))-u(\gamma(a))|\leq \int_{a}^{b}g(\gamma(t)) dt$ for all 1-Lipschitz curve $\gamma:[a,b]\rightarrow X$ \footnote{Since every rectifiable curve admits an arc-length parametrization that makes the  curve 1-Lipschitz, the class of 1-Lipschitz curves coincides with the class of rectifiable curves, modulo a parameter change.}.
\end{dfn} 
\begin{rem} If $X$ is a Riemannian manifold, $|\nabla u|$ is an upper gradient of $u\in \Lip$ and $|\nabla u|\leq g$ for all upper gradients $g$ of $u$.
\end{rem}
\begin{dfn}
For every locally Lipschitz continuous function $u$ defined on a open set of $X$, we define 
\begin{equation*}
\Lip u(x)= \begin{cases}
\lim\sup_{\stackrel {y\rightarrow x } {y\neq x}} \frac{|u(y)-u(x)|}{d(y,x)} \;\textrm{if }\, x \textrm{ is not isolated},
\\ 
0 \textrm{  otherwise.}
\end{cases}
\end{equation*}
\end{dfn}
\begin{rem} $\Lip u$ is an upper gradient of $u$.
\end{rem}
\begin{dfn}[Poincar\'{e} Inequality]\label{PG} A metric-measure space $(X,d,\mu)$ admits a weak local Poincar\'{e} inequality $(P_{qloc})$ for some $1\leq q<\infty$, if there exist $r_{1}>0,\,\lambda\geq 1,\;C=C(q,r_{1})>0$, such that for every continuous function $u$ and upper gradient $g$ of $u$, and for every ball $B$ of radius $0<r<r_{1} $ the following inequality holds:
\begin{equation*} \tag{$P_{qloc}$}
\Bigl(\aver{B}|u-u_{B}|^{q}d\mu\Bigr)^{\frac{1}{q}}\leq Cr \Bigl(\aver{\lambda B}g^{q}d\mu\Bigr)^{\frac{1}{q}}. 
\end{equation*}
If $\lambda=1$, we say that we have a strong local Poincar\'{e} inequality.
 \\ Moreover, $X$ admits a global Poincar\'{e} inequality or simply a Poincar\'{e} inequality $ (P_{q})$ if one can take $r_{1}=\infty$.
\end{dfn}
\subsection{Interpolation of the Sobolev spaces $H_{p}^{1}$} \label{ESM}
Before defining the Sobolev spaces $H_{p}^{1}$ it is convenient to recall the following proposition.
\begin{prop}\label{H}(see \cite{hajlasz3} and \cite{cheeger1} Theorem 4.38) Let $(X,d,\mu)$ be a  complete metric-measure space, with $\mu$ doubling and satisfying a  weak Poincar\'{e} inequality $(P_{q})$ for some $1<q<\infty$. Then there exist  an integer $N$, $C\geq 1$ and a linear operator $D$ which associates to each locally Lipschitz function $u$ a measurable function $Du\,: \,X\rightarrow \mathbb{R}^{N}$ such that :
\begin{itemize}
\item[1.] if $u\,$ is $ L$-Lipschitz, then $|Du|\leq CL\;\mu-a.e.$;
\item[2.] if $u$ is locally Lipschitz and constant on a measurable  set $E\subset X$, then $Du=0\; \mu -a.e.$ on $E$;
\item[3.] for locally Lipschitz functions $u$ and $v$, $D(uv)=uDv+vDu$;
\item[4.] for each locally Lipschitz function $u$, $\Lip u\leq |Du|\leq C\,\Lip u$, and hence $(u,|Du|)$ satisfies the weak  Poincar\'{e} inequality $(P_{q})$ .
\end{itemize}
\end{prop}
 We define now $H_{p}^{1}=H_{p}^{1}(X,d,\mu)$  for $1\leq p<\infty$ as the closure of locally Lipschitz functions for the norm\\
$$ 
\|u\|_{H_{p}^{1}}=\| u\|_{p}+\| \,|Du|\,\|_{p} \equiv \|u\|_{p}+\| \Lip u\|_{p}.
$$ 
We denote $H_{\infty}^{1}$ for the set of all bounded Lipschitz functions on $X$. 
\begin{rem} Under the hypotheses of Proposition \ref{H}, the uniqueness of the gradient holds for every $f\in H_{p}^{1}$ with $p\geq q$. By uniqueness of gradient we mean that if $u_{n}$ is a locally Lipschitz sequence such that  $u_n\rightarrow 0$ in $L_{p}$ and $Du_{n}\rightarrow g\in L_p$ then $g=0 \; a.e.$. Then $D$ extends to a bounded linear operator from $H_{p}^{1}$ to $L_{p}$.
\end{rem}
In the remaining part of this section, we consider a complete non-compact metric-measure space $(X,d,\mu)$ with $\mu$ doubling. We also assume that $X$ admits a Poincar\'{e} inequality $(P_{q}) $ for some $1<q<\infty$ as defined in Definition \ref{PG}.
By \cite{keith2} Theorem 1.3.4, this is equivalent to say that there exists $C>0$ such that for all $f\in \Lip$ and for every ball $B$ of $X$ of radius $r>0$ we have
\begin{equation*}\tag{$P_{q}$}
\int_{B}|f-f_{B}|^{q}d\mu\leq Cr^{q}\int_{B}|\Lip f|^{q}d\mu .
\end{equation*}

Define $q_{0}=\inf\left\lbrace q \in ]1,\infty[: (P_{q}) \textrm{ holds }\right\rbrace$.
 \begin{lem}\label{dL} Under these hypotheses, and for $q_{0}<p<\infty$, $Lip\cap H_{p}^{1}$ is dense in $H_{p}^{1}$.
\end{lem}
\begin{proof} See the proof of Theorem 9 in \cite{franchi1}.
\end{proof}
 \begin{prop}\textbf{Calder\'{o}n-Zygmund lemma for Sobolev functions}\\Let $(X,d,\mu)$ be a complete non-compact metric-measure space with $\mu$ doubling, admitting a Poincar\'{e} inequality $(P_{q})$ for some $1< q<\infty$.
Then, the Calder\'{o}n-Zygmund decomposition of Proposition \ref{CZ} still holds in the present situation for $f\in \Lip\cap H_{p}^{1}$, $q\leq p
<\infty$, replacing $\nabla f$ by $Df$.
\end{prop}
\begin{proof} The proof is similar, replacing $\nabla f$ by $Df$, using that $D$ of Proposition \ref{H} is linear. Since the $\chi_{i}$ are $\frac{C}{r_{i}}$ Lipschitz then $\| D\chi_{i}\|_{\infty}\leq \frac{C}{r_{i}}$ by item 1. of Theorem \ref{H} and the $b_{i}$'s are Lipschitz. We can see that $g$ is also Lipschitz. Moreover, using the finite additivity of $D$ and the property 2. of Proposition \ref{H}, we get the equality $\mu-a.e.$ 
 $$
 Dg= Df-D(\sum_{i}b_{i})=Df-(\sum_{i}Db_{i}).
 $$
The rest of the proof goes as in Proposition \ref{CZ}.
\end{proof}
\begin{thm}\label{EKM} Let $(X,d,\mu)$ be a complete non-compact metric-measure space with $\mu$ doubling, admitting a Poincar\'{e} inequality $(P_{q})$ for some $1< q<\infty$. Then, there exist $C_{1},\,C_{2}$ such that  for all $ f \in
 H_{q}^{1}+H_{\infty}^{1}$ and all $t>0$  we have
\begin{equation*} \tag{$\ast_{\textrm{met}}$}
C_{1} t^{\frac{1}{q}}\Bigl(|f|^{q **\frac{1}{q}}(t)+|Df|^{q**\frac{1}{q}}(t)\Bigr)\leq
K(f,t^{\frac{1}{q}},H_{q}^{1},H_{\infty}^{1})\leq C_{2}t^{\frac{1}{q}}
\Bigl(|f|^{q**\frac{1}{q}}(t)+|D
f|^{q**\frac{1}{q}}(t)\Bigr).
\end{equation*}
\end{thm}
\begin{proof} We have $(\ast_{\textrm{met}})$ for all $f \in \Lip\cap  H_{q}^{1}$ from the Calder\'{o}n-Zygmund decomposition that we have done.
Now for $f \in  H_{q}^{1}$, by Lemma \ref{dL}, $f=\lim\limits_{n}f_{n}$ in $ H_{q}^{1}$, with $f_{n}$ Lipschitz and $\| f-f_{n}\|_{ H_{q}^{1}}<\frac{1}{n}$ for all $n$.
Since for all $n$, $f_{n}\in \Lip$, there exist $g_{n},\,h_{n}$ such that $f_{n}=h_{n}+g_{n}$ and 
$\|h_{n}\|_{ H_{q}^{1}}+t^{\frac{1}{q}} \| g_{n}\|_{H_{\infty}^{1}} \leq Ct^{\frac{1}{q}}\Bigl(|f_{n}|^{q **\frac{1}{q}}(t)+|D f_{n}|^{q**\frac{1}{q}}(t)\Bigr)$. Therefore we find
\begin{align*}
\| f-g_{n}\|_{ H_{q}^{1}}+t^{\frac{1}{q}} \|g_{n}\|_{H_{\infty}^{1}}&\leq \| f-f_{n}\|_{ H_{q}^{1}}+(\| h_{n}\|_{ H_{q}^{1}}+t^{\frac{1}{q}} \| g_{n}\|_{H_{\infty}^{1}})
\\
&\leq \frac{1}{n}+Ct^{\frac{1}{q}}\Bigl(|f_{n}|^{q **\frac{1}{q}}(t)+|D f_{n}|^{q\,**\frac{1}{q}}(t)\Bigr).
\end{align*}
Letting $n\rightarrow \infty$, since $|f_{n}|^{q}\underset{n\rightarrow \infty}{\longrightarrow}|f|^{q}$ in $L_{1}$ and $|D f_{n}|^{q}\underset{n\rightarrow \infty}{\longrightarrow}|D f|^{q}$ in $L_{1}$, it comes  $
|f_{n}|^{q **}(t)\underset{n\rightarrow \infty}{\longrightarrow} |f|^{q **}(t)$ and $|Df_{n}|^{q **}(t)\underset{n\rightarrow \infty}{\longrightarrow} |D f|^{q **}(t)$ for all $t>0$. Hence 
$(\ast_{\textrm{met}})$ holds for $ f\in  H_{q}^{1}$. We prove $(\ast_{\textrm{met}})$ for $ f\in  H_{q}^{1}+H_{\infty}^{1}$ by the same argument of \cite{devore1}.
\end{proof}
\begin{thm}[Interpolation Theorem]\label{IHM} Let $(X,d,\mu)$ be a complete non-compact metric-measure space with $\mu$ doubling, admitting a Poincar\'{e} inequality $(P_{q})$ for some $1< q<\infty$. Then, for $ q_{0}<p_{1}<p<p_{2}\leq \infty$\footnote{ We allow $p_{1}=1$ if $q_{0}=1$.}, $ H_{p}^{1}$ is an interpolation space between $H_{p_{1}}^{1}$ and $H_{p_{2}}^{1}$.
\end{thm}
\begin{proof} Theorem \ref{EKM} provides us with all the tools needed for interpolating, as we did in the Riemannian case. In particular, we get Theorem \ref{IHM}.
\end{proof}
%We obtain then all the results of Section \ref{TI}, but now for metric-measure spaces.
\begin{rem} We were not able to get our interpolation result as in the Riemmanian case for $p_1\leq q_{0}$. Since we do not have Poincar\'e inequality $(P_{p_1})$, the uniqueness of the gradient $D$ does not hold in general in $H_{p_1}^{1}$.
\end{rem} 
\begin{rem} Other Sobolev spaces on metric-measure spaces were introduced in the last few years, for instance  $M_{p}^{1}$, $\,N_{p}^{1}$, $\,C_{p}^{1}$, $\,P_{p}^{1}$. If $X$ is a complete metric-measure space satisfying $(D)$ and $(P_{q})$ for some $1<q<\infty$, it can be shown that for $q_{0}<p\leq\infty$, all the mentioned spaces are equal to $H_{p}^{1}$ with equivalent norms (see \cite{hajlasz4}). In conclusion our interpolation result carries over to those Sobolev spaces.
\end{rem}
\begin{rem} 
The purpose of this remark is to extend our results to local assumptions. Assume that $(X,d,\mu)$ is a complete metric-measure space, with $\mu$ locally doubling, and admitting a local Poincar\'{e} inequality $(P_{qloc})$ for some $1<q<\infty$. Since $X$ is complete and $(X,\mu)$ satisfies a local doubling condition and a local Poincar\'{e} inequality  $(P_{qloc})$, then according to an observation of David and Semmes (see the introduction in \cite{cheeger1}), every ball $B(z,r)$, with $0<r<\min(r_{0},r_{1})$, is $\lambda=\lambda(C(r_{0}), C(r_{1}))$ quasi-convex, $C(r_{0})$ and $C(r_{1})$ being the constants appearing in the local doubling property and in the local Poincar\'{e} inequality. Then, for $0<r<\min(r_{0},r_{1})$, $B(z,r)$ is $\lambda$ bi-Lipschitz to a length space 
(one can associate, canonically, to a $\lambda$-quasi-convex metric space a length metric space, which is $\lambda$-bi-Lipschitz to the original one). Hence, we get a result similar to the one in Theorem \ref{EKM}. Indeed, the proof goes as that of Theorem \ref{EK} in the local case noting that the $B^{j}$'s considered there are then $\lambda$ bi-Lipschitz to a length space with $\lambda$ independent of $j$. Thus Lemma \ref{DB} still holds (see Remark \ref{RB}). Therefore, we get the characterization $(\ast_{\textrm{met}})$ of $K$  and by interpolating, we obtain the correspondance analogue of Theorem \ref{IHM}.
\end{rem}
\section{Applications}
\subsection{Carnot-Carath\'{e}odory spaces}
An important application of the theory of So\-bo\-lev spaces on metric-measure spaces is to a Carnot-Carath\'eodory space.  
We refer to \cite{hajlasz4} for a survey on the theory of Carnot-Carath\'{e}odory spaces. \\
Let $\Omega \subset \mathbb{R}^{n}$  be a connected open set, $X=(X_{1},...,X_{k})$ a family of vector fields defined on $\Omega$, with real locally Lipschitz continuous coefficients and $|Xu(x)|=\Bigl(\sum\limits_{j=1}^{k}|X_{j}u(x)|^{2}\Bigr)^{\frac{1}{2}}$.
We equip $\Omega$ with the Lebesgue measure $\mathcal{L}^{n}$ and the Carnot-Carath\'{e}odory metric $\rho$ associated to the $X_{i}$. We assume that $\rho$ defines a distance. Then, the metric space $(\Omega,\rho)$ is a length space.
\begin{dfn} Let $1\leq p<\infty$. We define $H_{p,X}^{1}(\Omega)$ as the completion of locally metric \footnote{that is relative to the metric $\rho$ of Carnot-Carath\'eodory.} Lipschitz functions (equivalently of $C^{\infty}$ functions ) for the norm 
$$
\|f\|_{H_{p,X}^{1}}=\| f \|_{L_{p}(\Omega)}+\|\,|Xf|\, \|_{L_{p}(\Omega)}
$$
\end{dfn} 
We denote $H_{\infty,X}^{1}$ for the set of bounded metric Lipschitz function.
\begin{rem} For all $1\leq p\leq \infty$, $ H_{p,X}^{1}= W_{p,X}^{1}(\Omega):=\left\lbrace f\in L_{p}(\Omega):\,|Xf|\in L_{p}(\Omega)\right\rbrace$, where  $Xf$ is defined in the distributional sense (see for example \cite{garofalo} Lemma 7.6).
\end{rem}
Adapting the same method, we obtain the following interpolation theorem for the $H_{p,X}^{1}$.
\begin{thm}\label{EKC}
Consider $(\Omega,\rho,\mathcal{L}^{n})$ where $\Omega$ is a connected open subset of $\mathbb{R}^{n}$. We assume that $\mathcal{L}^{n}$ is locally doubling, that the identity map $id:(\Omega,\rho)\rightarrow (\Omega,|.|)$ is an homeomorphism. Moreover, we suppose that the space admits a local weak Poincar\'{e} inequality $(P_{qloc})$ for some  $1\leq q< \infty$. Then, for $1\leq p_{1}<p<p_{2}\leq \infty$ with $p>q_0$, $H_{p,X}^{1}$ is an interpolation space between $H_{p_{1},X}^{1}$ and $H_{p_{2},X}^{1}$.
\end{thm}
\subsection{Weighted Sobolev spaces}
We refer to \cite{heinonen2}, \cite{kilpelainen} for the definitions used in this subsection. 
Let $\Omega$ be an open subset of $\mathbb{R}^{n}$ equipped with the Euclidean distance, $w\in L_{1,loc}(\mathbb{R}^{n})$ with $w>0,\,d\mu=wdx$. We assume that $\mu$  is $q$-admissible for some $1<q<\infty$ (see \cite{heinonen3} for the definition). This is equivalent to say, (see \cite{hajlasz4}), that $\mu$ is doubling and  there exists $C>0$ such that for every ball $B\subset\mathbb{R}^{n}$ of radius $r>0$ and for every function $\varphi \in C^{\infty}(B)$,
\begin{equation*}\tag{$P_{q}$}
\int_{B} |\varphi-\varphi_{B}|^{q}d\mu \leq Cr^{q} \int_{B}|\nabla \varphi|^{q}d\mu
\end{equation*}
with $\varphi_{B}=\frac{1}{\mu(B)}\int_{B}\varphi d\mu$. The $A_{q}$ weights, $q>1$, satisfy these two conditions (see \cite{heinonen3}, Chapter 15).
\begin{dfn}\label{DH} For $q\leq p<\infty$, we define the Sobolev space $H_{p}^{1}(\Omega,\mu)$ to be the closure of $C^{\infty}(\Omega)$ for the norm
$$
\|u\|_{H_{p}^{1}(\Omega,\mu)}=\|u\|_{L_{p}(\mu)}+\| \,|\nabla u| \,\|_{L_{p}(\mu)}.
$$
\end{dfn}
\noindent We denote $H_{\infty}^{1}(\Omega,\mu)$ for the set of all bounded Lipschitz functions on $\Omega$.
 \\
 
 Using our method, we obtain the following interpolation theorem for the  Sobolev spaces $H_{p}^{1}(\Omega,\mu)$:
\begin{thm}\label{ISw} Let $\Omega$ be as in above. Then for $q_{0} <p_{1}<p<p_{2}\leq\infty$, $H_{p}^{1}(\Omega,\mu)$ is an interpolation space between $H_{p_{1}}^{1}(\Omega,\mu)$ and  $H_{p_{2}}^{1}(\Omega,\mu)$.
\end{thm}
As in section 7, we were not able to get our interpolation result for $p_{1}\leq q_{0}$ since again in this case the uniqueness of the gradient does not hold for $p_{1}\leq q_{0}$. 
\begin{rem} Equip $\Omega$ with the Carnot-Carath\'{e}odory distance associated to a family of vector fields with real locally Lipschitz continuous coefficients instead of the Euclidean distance. Under the same hypotheses used in the beginning of this section, just replacing the balls $B$ by the balls $\tilde{B}$ with respect to $\rho$, and $\nabla$ by $X$ and assuming that $id:(\Omega,\rho)\rightarrow (\Omega,|.|)$ is an homeomorphism, we obtain our interpolation result. As an example we take vectors fields satisfying a H\"{o}rmander condition or vectors fields of Grushin type \cite{franchi5}.
 \end{rem}
\subsection{Lie Groups}
 In all this subsection, we consider $G$ a connected unimodular Lie group equipped with a Haar measure $d\mu$ and a family of left invariant vector fields $X_{1},...,X_{k}$ such that the $X_{i}$'s satisfy a H\"{o}rmander condition. In this case the Carnot-Carath\'{e}odory metric $\rho$ is is a distance, and $G$ equipped with the distance  $\rho$ is complete and defines the same topology as that of $G$ as a manifold (see \cite{coulhon8} page 1148). 
From the results in \cite{guivarch}, \cite{nagel}, it is known that $G$ satisfies $(D_{loc})$. Moreover, if $G$ has polynomial growth it satisfies $(D)$.  From the results in \cite{saloff2}, \cite{varopoulos2}, $G$ admits a local Poincar\'e inequality $(P_{1loc})$. If $G$ has polynomial growth, then it admits a global Poincar\'{e} inequality $(P_{1})$.
 \\
 
 \noindent \paragraph{\textbf{Interpolation of non-homogeneous Sobolev spaces.}} We define the non-ho\-mo\-ge\-neous Sobolev spaces on a Lie group $W_{p}^{1}$ in the same manner as in section 3 on a Riemannian manifold replacing $\nabla$ by $X$ (see Definition \ref{DNH} and Proposition \ref{CDW}). 
 \\ 

To interpolate  the $W_{p_{i}}^{1}$, we distinguish between the polynomial and the exponential growth cases. If $G$ has polynomial growth, then we are in the global case. If $G$ has exponential growth, we are in the local case. In the two cases we obtain the following theorem.
 \begin{thm} Let $G$ be as above. Then, for all $1\leq p_{1}<p<p_{2}\leq \infty,\;W_{p}^{1}$ is an interpolation space between $W_{p_{1}}^{1}$ and $W_{p_{2}}^{1},\,(q_{0}=1$ here). Therefore, we get all the interpolation theorems of section \ref{TI}.
\end{thm}

\noindent\paragraph{\textbf{Interpolation of homogeneous Sobolev spaces.}}
Let $G$ be a connected Lie group as before. We define the homogeneous Sobolev space $\dot{W}_{p}^{1} $ in the same manner as in section 5 on Riemannian manifolds replacing $\nabla$ by $X$.
\\

For these spaces we have the following interpolation theorem.
\begin{thm}  Let $G$ be as above and assume that $G$ has polynomial growth. Then for $1\leq p_{1}<p<p_{2}\leq \infty$, $ \dot{W}_{p}^{1}$ is an interpolation space between $\dot{W}_{p_{1}}^{1}$ and $\dot{W}_{p_{2}}^{1}$.
 \end{thm}

\section{Appendix} 
In view of the hypotheses in the previous interpolation results, a naturel question to address is whether the properties $(D)$ and $(P_{q})$ are necessary. The aim of the appendix is to exhibit an example where at least Poincar\'e is not needed.
Consider 
$$X=\left\lbrace(x_{1},x_{2},...,x_{n})\in \mathbb{R}^{n};\, x_{1}^{2}+...+x_{n-1}^{2}\leq x_{n}^{2} \right\rbrace
$$
equipped with the Euclidean metric of $\mathbb{R}^{n}$ and with the Lebesgue measure. $X$ consists of two infinite closed cones with a common vertex. $X$ satisfies the doubling property and admits $(P_{q})$ in the sense of metric-measure spaces if and only if $q>n$ (\cite{hajlasz4} p.17). Denote by $\Omega$ the interior of $X$. Let $H_{p}^{1}(X)$ be the closure of $\Lip_{0}(X)$ for the norm
$$
\|f\|_{H_{p}^{1}(X)}=\|f\|_{L_{p}(\Omega)}+\|\,|\nabla f|\,\|_{L_{p}(\Omega)}.
$$
We define $W_{p}^{1}(\Omega)$ as the set of all functions $f\in L_{p}(\Omega)$ such that $\nabla f \in L_{p}(\Omega)$ and equip this space with the norm
$$
\|f\|_{ W_{p}^{1}(\Omega)}=\|f\|_{H_{p}^{1}(X)}.
$$
The gradient is always defined in the distributional sense on $\Omega$.

Using our method, it is easy to check that the $W_{p}^{1}(\Omega)$ interpolate for all $1\leq p\leq \infty$. Also our interpolation result asserts that $H_{p}^{1}(X)$ is an interpolation space between $H_{p_{1}}^{1}(X)$ and $H_{p_{2}}^{1}(X)$ for $1\leq p_{1}<p<p_{2}\leq \infty$ with
 $p>n$. It can be shown that $H_{p}^{1}(X)\subsetneq W_{p}^{1}(\Omega)$ for $p>n$ and $H_{p}^{1}(X)= W_{p}^{1}(\Omega)$ for $1\leq p<n$. Hence $H_{p}^{1}(X)$ is an interpolation space between $H_{p_{1}}^{1}(X)$ and $H_{p_{2}}^{1}(X)$ for $1\leq p_{1}<p<p_{2}<n$ although the Poincar\'{e} inequality does not hold on $X$ for those $p$. 
%In this way, we see that Poincar\'e inequality is not necessary in general.
However, we do not know if the $H_{p}^{1}$ interpolate for all $1\leq p\leq \infty$ (see \cite{badr}, Chapter 4 for more details).
\\

\nocite{mazya}
\bibliographystyle{plain}
\bibliography{latex2}

\begin{thebibliography}{10}

\bibitem{ambrosio1}
L.~Ambrosio, M.~Miranda~Jr, and D.~Pallara.
\newblock Special functions of bounded variation in doubling metric measure
  spaces.
\newblock {\em Calculus of variations: topics from the mathematical heritage of
  E. De Giorgi, Quad. Mat., Dept. Math, Seconda Univ. Napoli, Caserta},
  14:1--45, 2004.

\bibitem{aubin1}
T.~Aubin.
\newblock Espaces de {S}obolev sur les vari\'{e}t\'{e}s {R}iemanniennes.
\newblock {\em Bull. Sci. Math. 2}, 100(2):149--173, 1976.

\bibitem{auscher1}
P.~Auscher and T.~Coulhon.
\newblock Riesz transform on manifolds and {P}oincar\'{e} inequalities.
\newblock {\em Ann. Scuola Norm. Sup. Pisa Cl.Sci(5)}, 4(3):531--555, 2005.

\bibitem{auscher2}
P.~Auscher, T.~Coulhon, X.T. Duong, and S.~Hofmann.
\newblock Riesz transform on manifolds and heat kernel regularity.
\newblock {\em Ann. Sci. Ecole Norm. Sup.}, 37:911--957, 2004.

\bibitem{badr}
N.~Badr.
\newblock {\em Ph.D Thesis}.
\newblock Universit\'e Paris-Sud 11, 2007.

\bibitem{bennett}
C.~Bennett and R.~Sharpley.
\newblock {\em Interpolations of operators}.
\newblock Academic Press, 1988.

\bibitem{bergh}
J.~Bergh and J.~L\"{o}fstr\"{o}m.
\newblock {\em Interpolations spaces, {A}n introduction}.
\newblock Springer (Berlin), 1976.

\bibitem{calderon2}
A.~P. Calder\'{o}n.
\newblock Spaces between ${L}^{1}$ and ${L}^{\infty}$ and the theorem of
  {M}arcinkiewicz.
\newblock {\em Studia Math.}, 26:273--299, 1966.

\bibitem{calderon1}
C.~P. Calder\'{o}n and M.~Milman.
\newblock Interpolation of {S}obolev {S}paces. {T}he {R}eal {M}ethod.
\newblock {\em Indiana. Univ. Math. J.}, 32(6):801--808, 1983.

\bibitem{cheeger1}
J.~Cheeger.
\newblock Differentiability of {L}ipschitz functions on metric measure spaces.
\newblock {\em Geom. Funct. Anal.}, 9:428--517, 1999.

\bibitem{coifman2}
R.~Coifman and G.~Weiss.
\newblock {\em Analyse harmonique sur certains espaces homog\`{e}nes}.
\newblock Lecture notes in Math., Springer, 1971.

\bibitem{coifman1}
R.~Coifman and G.~Weiss.
\newblock Extensions of {H}ardy spaces and their use in analysis.
\newblock {\em Bull. Amer. Math. Soc.}, 83:569--645, 1977.

\bibitem{coulhon8}
T.~Coulhon, I.~Holopainen, and L.~Saloff~Coste.
\newblock Harnack inequality and hyperbolicity for the subelliptic $p$
  {L}aplacians with applications to {P}icard type theorems.
\newblock {\em Geom. Funct. Anal.}, 11(6):1139--1191, 2001.

\bibitem{devore1}
R.~Devore and K.~Scherer.
\newblock Interpolation of linear operators on {S}obolev spaces.
\newblock {\em Ann. of Math.}, 109:583--599, 1979.

\bibitem{franchi3}
B.~Franchi.
\newblock Weighted {S}obolev-{P}oincar\'{e} inequalities and pointwise
  estimates for a class of degenerate elliptic equations.
\newblock {\em Trans. Amer. Math. Soc}, 327(1):125--158, 1991.

\bibitem{franchi5}
B.~Franchi, C.E. Guti\'{e}rrez, and R.L. Wheeden.
\newblock Weighted {S}obolev-{P}oincar\'{e} inequalities for {G}rushin type
  operators.
\newblock {\em Com. Partial Differential Equations}, 19:523--604, 1994.

\bibitem{franchi1}
B.~Franchi, P.~Hajlasz, and P.~Koskela.
\newblock Definitions of {S}obolev classes on metric spaces.
\newblock {\em Ann. Inst. Fourier (Grenoble)}, 49:1903--1924, 1999.

\bibitem{franchi7}
B.~Franchi, F.~Serrapioni, and F.~Serra~Cassano.
\newblock Approximation and imbedding theorems for weighted {S}obolev {S}paces
  associated with {L}ipschitz continuous vector fields.
\newblock {\em Boll. Un. Math. Ital.}, B(7):83--117, 1997.

\bibitem{garofalo}
N.~Garofalo and D.~M. Nhieu.
\newblock Isoperimetric and {S}obolev {I}nequalities for
  {C}arnot-{C}arathéodory spaces and the existence of minimal surfaces.
\newblock {\em Comm. Pure Appl. Math.}, 49:1081--1144, 1996.

\bibitem{goldshtein}
V.~Gol'dshtein and M.~Troyanov.
\newblock Axiomatic {T}heory of {S}obolev {S}paces.
\newblock {\em Expo. Mathe.}, 19:289--336, 2001.

\bibitem{guivarch}
Y.~Guivarc'h.
\newblock Croissance polynomiale et p\'{e}riode des fonctions harmoniques.
\newblock {\em Bull. Soc. Math. France}, 101:149--152, 1973.

\bibitem{hajlasz3}
P.~Hajlasz.
\newblock Sobolev spaces on metric measure spaces (heat kernels and analysis on
  manifolds, graphs, and metric spaces).
\newblock {\em Contemp. Math., Amer. Math. Soc.}, (338):173--218, 2003.

\bibitem{hajlasz4}
P.~Hajlasz and P.~Koskela.
\newblock {S}obolev met {P}oincar\'{e}.
\newblock {\em Mem. Amer. Math. Soc.}, 145(688):1--101, 2000.

\bibitem{heinonen2}
J.~Heinonen.
\newblock {\em Lectures on analysis on metric spaces}.
\newblock Springer-Verlag, 2001.

\bibitem{heinonen3}
J.~Heinonen.
\newblock {\em Non smooth calculus}.
\newblock Memoirs of A.M.S., 2007.

\bibitem{heinonen5}
J.~Heinonen and P.~Koskela.
\newblock Quasiconformal maps in metric spaces with controlled geometry.
\newblock {\em Acta Math.}, 181:1--61, 1998.

\bibitem{keith2}
S.~Keith and K.~Rajala.
\newblock A remark on {P}oincar\'{e} inequality on metric spaces.
\newblock {\em Math. Scand.}, 95(2):299--304, 2004.

\bibitem{keith3}
S.~Keith and X.~Zhong.
\newblock The {P}oincar\'{e} inequality is an open ended condition.
\newblock {\em To appear in Ann. of Math.}

\bibitem{kilpelainen}
T.~Kilpel\"{a}inen.
\newblock Smooth approximation in {W}eighted {S}obolev spaces.
\newblock {\em Comment. Math. Univ. Carolinae}, 38(1):29--35, 1997.

\bibitem{martin}
J.~Mart\'{i}n and M.~Milman.
\newblock Sharp {G}agliardo-{N}irenberg inequalities via symmetrization.
\newblock {\em Math. Res. Lett.}, 14(1):49--62, 2007.

\bibitem{mazya}
V.~G. Maz'ya.
\newblock {\em {S}obolev spaces}.
\newblock Springer (Berlin), 1985.

\bibitem{nagel}
A.~Nagel, E.~M. Stein, and S.~Wainger.
\newblock Balls and metrics defined by vector fields.
\newblock {\em Acta Math.}, 155:103--147, 1985.

\bibitem{saloff2}
L.~Saloff-Coste.
\newblock Parabolic {H}arnack inequality for divergence form second order
  differential operator.
\newblock {\em Potential Anal.}, 4(4):429--467, 1995.

\bibitem{stein3}
E.~M. Stein and G.~Weiss.
\newblock {\em Introduction to {F}ourier {A}nalysis in {E}uclidean spaces}.
\newblock Princeton University Press, 1971.

\bibitem{varopoulos2}
N.~Varopoulos.
\newblock Fonctions harmoniques sur les groupes de {L}ie.
\newblock {\em C. R. Acad. Sc. Paris, Ser. I}, 304(17):519--521, 1987.

\end{thebibliography}
\end{document}